\newcommand{\showdate}{false}

\documentclass[reqno,a4paper]{amsart}
\usepackage{amssymb}
\usepackage[utf8]{inputenc}
\usepackage{a4wide}
\usepackage{enumitem}
\usepackage{pstricks}
\usepackage{ifthen}
\usepackage{comment}
\usepackage[all]{xy}

\title{New invariants of $G_2$-structures}
\author[D. Crowley]{Diarmuid Crowley}
\address{Institute of Mathematics,
University of Aberdeen,
Aberdeen AB24 3UE, UK}
\email{dcrowley@abdn.ac.uk}
\author[J. Nordström]{Johannes Nordström}
\address{Department of Mathematical Sciences,
University of Bath,
Bath BA2 7AY, UK}
\email{j.nordstrom@bath.ac.uk}

\ifthenelse{\boolean{\showdate}}{\date{\today}}{}

\usepackage{color}

\newcommand{\ignore}[1]{}

\DeclareMathOperator{\re}{Re}

\DeclareMathOperator{\im}{Im}

\DeclareMathOperator{\vol}{vol}

\DeclareMathOperator{\lcm}{lcm}

\DeclareMathOperator{\Sec}{Sec}
\DeclareMathOperator{\Clo}{Clo}
\DeclareMathOperator{\Tor}{Tor}
\newcommand{\ie}{\emph{i.e.} }
\newcommand{\eg}{\emph{e.g.} }
\newcommand{\cf}{\emph{cf.} }
\newcommand{\sign}{\sigma}
\newcommand{\tmfd}{Y}
\newcommand{\mpl}{m}
\newcommand{\cd}{k}
\newcommand{\cm}{p}

\newcommand{\pll}{\pi}
\newcommand{\hz}{h}
\newcommand{\td}{\tilde d}
\newcommand{\tdf}{\textstyle \frac{\td_\pi}{4}} 
\newcommand{\EK}{\mu}

\newcommand{\ahat}{\widehat{A}}
\newcommand{\dpi}[1]{d_\pi(#1)}
\newcommand{\mdiv}{d_o}
\newcommand{\dinfty}[1]{\mdiv(#1)}
\newcommand{\rdbar}{\widehat\varphi_{rd}}
\newcommand{\nubar}{\bar\nu}

\newcommand{\mmod}{\!\!\mod}

\newcommand{\half}{{\textstyle\frac{1}{2}}}

\newcommand{\quart}{\frac{1}{4}}
\newcommand{\threequart}{{\frac{3}{4}}}

\newcommand{\hoi}[1]{\left[#1\right)}

\newcommand{\bbz}{\mathbb{Z}}
\newcommand{\Z}{\mathbb{Z}}
\newcommand{\Q}{\mathbb{Q}}
\newcommand{\bbr}{\mathbb{R}}

\newcommand{\bbc}{\mathbb{C}}
\newcommand{\bbh}{\mathbb{H}}
\newcommand{\bbo}{\mathbb{O}}

\newcommand{\trivr}{\underline{\bbr}}
\newcommand{\spinr}{\Delta}
\newcommand{\spinrp}{\Delta^+}
\newcommand{\spinrn}{\Delta^-}
\newcommand{\spinrpm}{\Delta^\pm}
\newcommand{\spinrmp}{\Delta^\mp}
\newcommand{\spinb}{S}
\newcommand{\spinbp}{\spinb^+}
\newcommand{\spinbn}{\spinb^-}
\newcommand{\spinbpm}{\spinb^\pm}
\newcommand{\spins}{s}
\newcommand{\rcliff}{\widehat c}
\newcommand{\kd}{\Sigma}

\newcommand{\into}{\hookrightarrow}

\newcommand{\Diff}{\textup{Diff}}
\newcommand{\DiffSpin}{\Diff}
\newcommand{\ADiff}{\textup{ADiff}}
\DeclareMathOperator{\aut}{Aut}
\newcommand{\Num}{\textup{Num}}
\newcommand{\NumB}[1]{\Num\left(#1\right)}

\newcommand{\Id}{\textup{Id}}
\newcommand{\del}{\partial}

\newcommand{\calr}{\mathcal{R}}
\newcommand{\calv}{\mathcal{V}}

\newcommand{\GT}{\mathcal{G}_2}
\newcommand{\GTB}{\mathcal{\bar G}_2}
\newcommand{\GTH}{\pi_0\GT}
\newcommand{\GTD}{\pi_0\GTB}
\newcommand{\GTC}{\GT^{cc}}

\newcommand{\gtstr}{$G_{2}$\nobreakdash-\hspace{0pt}structure}
\newcommand{\gtstrsoft}{$G_{2}$-structure}

\newcommand{\gtmfd}{$G_{2}$\nobreakdash-\hspace{0pt}manifold}
\newcommand{\gtmfdsoft}{$G_{2}$-manifold}

\newcommand{\gtmetric}{$G_2$ metric}
\newcommand{\spstr}{$Spin(7)$-structure}
\newcommand{\sustr}{$SU(3)$-structure}

\newcommand{\gl}[1]{GL(#1, \bbr)}

\newcommand{\co}{\mathcal{O}}
\newcommand{\ow}{{\overline W}}
\newcommand{\anglen}{u}
\newcommand{\anglex}{v}
\newcommand{\wt}[1]{\widetilde #1}
\newcommand{\tm}{\mkern 4mu\widetilde{\mkern-4mu M\mkern-1mu}\mkern 1mu}
\newcommand{\tf}{\mkern 4mu\widetilde{\mkern-4mu F\mkern-1mu}\mkern 1mu}

\newcommand{\tv}{\tilde\varphi}
\newcommand{\tomega}{\tilde\omega}
\newcommand{\tOmega}{\tilde\Omega}

\newcommand{\gen}[1]{\langle#1\rangle}
\newcommand{\inner}[1]{\langle#1\rangle}

\newcommand{\contra}[1]{\textstyle\frac{\partial}{\partial #1}}

\newcommand{\divsp}{S_{d_\pi}}
 
\newtheorem{thm}{Theorem}[section]
\newtheorem{prop}[thm]{Proposition}
\newtheorem{lem}[thm]{Lemma}

\newtheorem{cor}[thm]{Corollary}

\theoremstyle{definition}
\newtheorem{defn}[thm]{Definition}

\theoremstyle{remark}
\newtheorem{rmk}[thm]{Remark}
\newtheorem{ex}[thm]{Example}

\setlist{leftmargin=*}
\begin{document}

\begin{abstract}
We define a $\bbz_{48}$-valued homotopy invariant $\nu(\varphi)$ of a \gtstr{}
$\varphi$ on the
tangent bundle of a closed $7$-manifold in terms of the signature and Euler
characteristic of a coboundary with a \spstr. For manifolds of holonomy $G_2$
obtained by the twisted connected sum construction, the associated torsion-free
\gtstr{} always has $\nu(\varphi) = 24$.
Some holonomy $G_2$ examples constructed by Joyce by desingularising orbifolds
have odd $\nu$.

We define a further homotopy invariant $\xi(\varphi)$ such
that if $M$ is 2-connected then the pair $(\nu,\xi)$ determines a
\gtstr{} up to homotopy and diffeomorphism.
The class of a \gtstr{} is determined by $\nu$ on its own when the greatest
divisor of $p_1(M)$ modulo torsion divides 224; this sufficient condition holds
for many twisted connected sum $G_2$-manifolds.


We also prove that the parametric h-principle holds for coclosed \gtstr s.
\end{abstract}

\maketitle

\ifthenelse{\boolean{\showdate}}{\vspace{-0.8\baselineskip}}{}
\vspace{-\baselineskip}

\section{Introduction} \label{sec:introduction}
In this paper we develop methods to determine when two
\gtstr s on a closed 7-manifold are deformation-equivalent, by which we mean
related by homotopies 
and diffeo\-morphisms. The main
motivation is to study the problem of deformation-equivalence of metrics with
holonomy~$G_2$. Such metrics can be defined in terms of torsion-free \gtstr s.
The torsion-free condition is a complicated PDE, but we ignore that and
consider only the \gtstr{} as a topological residue of the holonomy $G_2$
metric: for a pair of \gtmetric s to be deformation-equivalent, it is
certainly necessary that the associated \gtstr s are. One would not expect this
necessary condition to be sufficient since the torsion-free constraint is quite
rigid. A much weaker constraint on a \gtstr{} is for it to be coclosed, and we
find that the h-principle holds in this case: if two coclosed \gtstr s can
be connected by a path of \gtstr s then they can also be connected by a path of
coclosed \gtstr s.

\subsection{The $\nu$-invariant}
A \gtstr{} on a $7$-manifold $M$ is a reduction of the structure group
of the frame bundle of $M$ to the exceptional Lie group $G_2$.
As we review in \S \ref{subsec:strs}, a \gtstr{} on $M$ is equivalent to
a 3-form $\varphi \in \Omega^3(M)$ of a certain type and
we will therefore refer to such `positive' 3-forms as \gtstr s.
A \gtstr{} induces a Riemannian metric and spin structure on $M$.
Throughout this introduction $M$ shall be a closed connected spin
$7$-manifold and all \gtstr s $\varphi$ will be compatible with the
chosen spin structure.
We denote the space of all such \gtstr s by~$\GT(M)$.

We say that two \gtstr s are homotopic if they can be connected by a
continuous path of \gtstr s, so the set of homotopy classes
of \gtstr s on $M$ is $\GTH(M)$. 
The following observation is not new, but the closest statement we have found
in the literature is Witt \cite[Proposition 3.3]{witt06}. The proof is simple
and provided in \S \ref{subsec:D}. 

\begin{lem}  \label{lem:GTH=Z}
The group $H^7(M; \pi_7(S^7)) \cong \Z$ acts freely and transitively on $\GTH(M) \equiv \Z$.
%
%
\end{lem}

The group of spin diffeomorphisms of $M$, $\DiffSpin(M)$, acts by pull-back 
on $\mathcal{G}_2(M)$ with quotient $\GTB(M) : = \GT(M)/\DiffSpin(M)$.
Since $\GT(M)$ is locally path connected 
\[ \GTD(M) = \GTH(M)/\pi_0\DiffSpin(M), \]
and we call $\GTD(M)$ the set of deformation classes of \gtstr s on $M$.
Up until now neither invariants
of $\GTD(M)$ nor results about its cardinality have appeared in the literature.

Our starting point for studying both of these problems is the following
characteristic class formula, valid for any closed spin 8-manifold~$X$
(see Corollary \ref{cor:ahat}):
\begin{equation}
\label{eq:char}
e_+(X) = 24 \ahat(X) + \frac{\chi(X) - 3 \sign(X)}{2}.
\end{equation}
Here the terms are the integral of the Euler class of the positive
spinor bundle, and the $\ahat$-genus, ordinary Euler characteristic and
signature of $X$ ($\ahat(X)$ is an integer because $X$ is spin, and
$\sign(X) \equiv \chi(X) \mmod 2$ for any closed oriented $X$).
Moving from $Spin(8)$ to $Spin(7)$, if we use the (real dimension 8) spin
representation of $Spin(7)$ to regard $Spin(7)$ as a subgroup of $\gl8$, then a
$Spin(7)$-structure on an 8-manifold $X$ can be characterised by a certain kind
of 4-form $\psi \in \Omega^4(X)$.  
A $Spin(7)$-structure defines 
a spin structure and Riemannian metric on $X$, and (up to a sign) a unit spinor
field of positive chirality.
In particular, if a closed $8$-manifold $X$ has a $Spin(7)$-structure then
$e_+(X) = 0$, and \eqref{eq:char} implies
\begin{equation}
\label{eq:defect}
48\ahat(X) + \chi(X) -3\sign(X) = 0.
\end{equation}
If $W$ is a compact 8-manifold with boundary $M$ then a $Spin(7)$-structure on
$W$ induces a \gtstr{} on $M$. From \eqref{eq:defect} one deduces that the
``$\ahat$ defect'' $\chi(W) -3\sign(W) \mmod 48$
depends only on the induced \gtstr{} on $M$.
It turns out, see Lemma \ref{lem:spinbordism}, that any
\gtstr{} $\varphi$ on $M$ 
bounds a $Spin(7)$-structure on some compact 8-manifold and this allows us to
define an invariant $\nu(\varphi)$.

\begin{defn} \label{def:nu}
Let $(M, \varphi)$ be a closed spin $7$-manifold with \gtstr{} and
$Spin(7)$-coboundary $(W, \psi)$.  The $\nu$-invariant of $\varphi$ is the
residue 
\[ \nu(\varphi) : = \chi(W) - 3\sigma(W) \mod 48 ~~\in  \Z_{48}. \]
\end{defn}

This definition makes sense even if $M$ is not connected, and is
additive under disjoint unions.
Among the many analogous invariants in differential topology, perhaps the
one best known to non-topologists is Milnor's $\Z_7$-valued $\lambda$-invariant
of homotopy 7-spheres, defined as a ``$p_2$ defect'' of a spin
coboundary \cite{milnor56}.
To distinguish all 28 smooth structures on a
homotopy sphere one can use the Eells-Kuiper invariant $\mu$ \cite{eells62},
which is another $\ahat$ defect (see \eqref{eq:classic_ek}).

In \S\ref{subsec:Intro_affine} we describe how $\nu$ is related to Lemma
\ref{lem:GTH=Z} by interpreting \gtstr s in terms of spinor fields, and
we develop most of the theory in those terms.
However, the definition above is sometimes useful when dealing with
examples. It lets us compute $\nu$ from a
coboundary with the right type of 4-form, and finding such $4$-forms can be
easier than describing spinor fields directly, \eg in the proof of Theorem
\ref{thm:nu24} and Examples \ref{ex:rdnu} and \ref{ex:sq}.

Theorem \ref{thm:defnu} below summarises the basic properties of $\nu$.
Note that if $\varphi$ is a \gtstr{} on~$M$, then the 3-form $-\varphi$ is
also a \gtstr{}, but compatible with the \emph{opposite} orientation;
$-\varphi$~is a \gtstr{} on $-M$.
In addition, if $X$ is a closed $(2n{+}1)$--manifold, we define its rational
semi-characteristic by $\chi_\Q(X) := \sum_{i=0}^n b_i(X) \mmod 2$.

\begin{thm}
\label{thm:defnu}
For all \gtstr s $\varphi$ on $M$, 
$\nu(\varphi) \in \bbz_{48}$ is well-defined, and invariant under homotopies
and diffeomorphisms.  Hence $\nu$ defines a function
\begin{equation}
\nu : \GTD(M) \to \bbz_{48}.
\end{equation}
Moreover $\nu(-\varphi) = -\nu(\varphi)$, and
$\nu$ takes exactly the 24 values allowed by the parity constraint
\begin{equation} \label{eq:parity}
\nu(\varphi) \equiv \chi_\Q(M) \mod 2 .
\end{equation}
\end{thm}

Theorem \ref{thm:defnu} entails that $\GTD(M)$ has at least 24 elements.
Here are some related questions that motivate our investigations:

\begin{itemize}[labelindent=\parindent]
\item
What are the values of $\nu$ for torsion-free \gtstr s, \ie ones arising from
$G_2$ holonomy metrics?
Are there \gtmetric s on the same manifold that can be distinguished by $\nu$?

\item Do there exist \gtmetric s that are not deformation-equivalent, but whose
associated torsion-free \gtstr s belong to the same class in $\GTD(M)$?

\item
What is the cardinality of $\GTD(M)$?  For example, for which closed spin
manifolds $M$ is $\nu$ a complete invariant of $\GTD(M)$?
\end{itemize}
\noindent
We give partial answers to the first and third of these questions
below, and discuss directions for further research in~\S\ref{subsec:further}.

\subsection{The affine difference $D$, spinors and the $\nu$-invariant}
\label{subsec:Intro_affine}

An important feature of homotopy classes of \gtstr s is that the identification
$\GTH(M) \equiv \Z$ from Lemma \ref{lem:GTH=Z} should be regarded as affine,
or as a $\Z$-torsor:
there is no preferred base point, but 
Lemma \ref{lem:GTH=Z} has the following consequence.

\begin{lem}  \label{lem:defd}
For any pair of \gtstr s $\varphi, \varphi'$ on $M$ there is a difference
$D(\varphi, \varphi') \in \bbz$ such that
$(\GTH(M), \; D) \cong (\Z, \; \mathrm{subtraction})$, \ie
$D(\varphi, \varphi') = 0$ if and
only if $\varphi$ is homotopic to $\varphi'$, and
\begin{equation}
\label{eq:d_affine}
D(\varphi, \varphi') + D(\varphi', \varphi'') = D(\varphi, \varphi'') .
\end{equation}
\end{lem}

To understand the relationship between $D$ and $\nu$, we first explain the
reasoning which goes into the proof of Lemma \ref{lem:GTH=Z}.
As we describe in \S \ref{subsec:G2spinors}, a
choice of Riemannian metric and unit spinor field on the spin manifold $M$
defines a \gtstr. Because any two Riemannian metrics are homotopic, this sets up
a bijection between $\GTH(M)$ and homotopy classes of sections of the unit
spinor bundle. This is an $S^7$-bundle, and Lemma \ref{lem:GTH=Z} follows 
from obstruction theory for sections of sphere bundles.


We can both describe $D$ in concrete terms and prove Lemma \ref{lem:defd} by
counting zeros of homotopies of spinor fields (see \S \ref{subsec:D}).
With this understanding of $D$, the next lemma is elementary. The intuitive
notion of a $Spin(7)$-bordism is spelt out in \S \ref{subsec:spin7bordisms}.

\begin{lem}
\label{lem:d_speuler} 
Let $\varphi$, $\varphi'$ be \gtstr s on $M$. Suppose $(W,\psi)$ is a
$Spin(7)$-bordism from $(M, \varphi)$ to $(M, \varphi')$, and let $\ow$ be
the closed spin 8-manifold formed by identifying the two boundary components
(\cf \eqref{eq:close}). Then
\begin{equation}
\label{eq:d_speuler}
D(\varphi, \varphi') = -e_+(\ow) .
\end{equation}
\end{lem}

Combining Lemma \ref{lem:d_speuler} with the characteristic class formula \eqref{eq:char}, the
mod 24 residue of $D(\varphi, \varphi')$ can be
computed from just the signature and Euler characteristic of $\ow$, which
equal those of $W$. So while $D$ only makes sense as an ``affine'' invariant,
its mod 24 residue is related to the ``absolute'' invariant $\nu$
(in particular, $\nu$ is affine linear).

\begin{prop}
Let $\varphi$ and $\varphi'$ be \gtstr s on 
$M$. Then
\label{prop:d_nu}
\begin{equation}
\nu(\varphi') - \nu(\varphi) \equiv 2D(\varphi,\varphi') \mod 48 .
\end{equation}
\end{prop}


\subsection{The $\nu$-invariant for manifolds with $G_2$ holonomy}  \label{subsec:Intro_HolG_2}

The exceptional Lie group $G_2$ also occurs as an exceptional case in the
classification of Riemannian holonomy groups due to Berger \cite{berger55}.
It is immediate from the definitions that a metric on a $7$-manifold
$M$ has holonomy contained in $G_2$ if and only if it is induced by
a \gtstr{} $\varphi \in \Omega^3(M)$ that is parallel.
The covariant derivative $\nabla \varphi$ of $\varphi$ with respect to the
Levi--Civita connection $\nabla$ of its induced metric can be identified with
the intrinsic torsion of the \gtstr, so metrics with holonomy in $G_2$
correspond to torsion-free \gtstr s \cite[Corollary 2.2, \S 11]{salamon89}.

One can define a moduli space of torsion-free \gtstr s on a fixed closed
\gtmfd{} $M$, which is an orbifold locally homeomorphic to finite quotients
of $H^3_{dR}(M)$. But while the
local structure is well understood, little is known about the global structure.
One basic question is whether the moduli space is connected, \ie whether any
pair of torsion-free \gtstr s are equivalent up to homotopies through
torsion-free \gtstr s and diffeomorphism. If one could find examples of
diffeomorphic \gtmfd s where the associated \gtstr s have different values
of $\nu$, this would prove that the moduli space is disconnected.

Finding compact manifolds with holonomy $G_2$ is a hard problem.
The known constructions solve the non-linear PDE 
$\nabla \varphi = 0$ using gluing methods.
Joyce \cite{joyce96-I} found the first examples by desingularising flat
orbifolds, and later Kovalev \cite{kovalev03} implemented a `twisted connected
sum' construction. In \cite{g2m}, the classification theory of closed
\mbox{2-connected} \mbox{7-manifolds} is used to find examples of twisted
connected sum \gtmfd s that are diffeomorphic, but without
any evidence either way as to whether the torsion-free \gtstr s are in the same
component of the moduli space.

The twisted connected sum $G_2$-manifolds are constructed by gluing a pair of pieces of the
form $S^1 \times V$, where $V$ are asymptotically cylindrical Calabi--Yau
3-folds with asymptotic ends $\bbr \times S^1 \times K3$.  We review this
construction in \S \ref{subsec:tcs} and then compute $\nu$ for
all such \gtstr s.

\begin{thm}
\label{thm:nu24}
If $(M, \varphi)$ is a twisted connected sum 
then $\nu(\varphi) = 24$.
\end{thm}

We carry out this calculation 
by finding an explicit $Spin(7)$-bordism from
a twisted connected sum \gtstr{} $\varphi$ to a \gtstr{} that is a product of structures
on lower-dimensional manifolds, for which $\nu$ is easier to evaluate.

For all the explicit examples of pairs of diffeomorphic \gtmfd s found in
\cite{g2m}, Corollary \ref{cor:nuauto} below implies
that $\nu$ classifies the homotopy classes of \gtstr s up to diffeomorphism.
Thus diffeomorphisms between these \gtmfd s can always be chosen so that the
corresponding torsion-free \gtstr s are homotopic. Theorem \ref{thm:hp}
implies that they are then also homotopic as coclosed \gtstr s, but
the question whether they can be connected by a path of torsion-free \gtstr s,
so that they are in the same component of the moduli space of \gtmetric s,
remains open.

Theorem \ref{thm:nu24} does not necessarily apply to more general gluings of
asymptotically cylindrical \gtmfd s. For example, a small number of the
\gtmfd s $M$ constructed by Joyce \cite[\S 12.8.4]{joyce00} have
$\chi_\Q(M) = 1$, so those torsion-free \gtstr s have odd $\nu \not= 24$; yet
they can be regarded at least topologically as a gluing of asymptotically
cylindrical manifolds.

\subsection{The h-principle for coclosed \gtstrsoft s}

We call a \gtstr{} with defining 3-form $\varphi$ \emph{closed} if
$d\varphi = 0$ and \emph{coclosed} if $d^*\varphi = 0$, where $d^*$ is defined
in terms of the metric induced by the \gtstr. For $\varphi$ to be torsion-free
is equivalent to it being both closed and coclosed (Fern\'andez--Gray
\cite{fernandez82}). Individually, the conditions
of being closed or coclosed are much more flexible than the torsion-free
condition, and we show that coclosed \gtstr s satisfy the h-principle.
Let $\GTC(M) \subset \GT(M)$ be the subspace of coclosed \gtstr s.

\begin{thm}
\label{thm:hp}
The inclusion $\GTC(M) \into \GT(M)$ is a homotopy equivalence.
\end{thm}


If $M$ is an open manifold then Theorem \ref{thm:hp} is a straight-forward
application of Theorem 10.2.1 from Eliashberg--Mishachev \cite{eliashberg02}
(\cf Lê \cite[Theorem-Remark 3.17]{le06}). h-principles are generally much
harder to prove on closed manifolds, but for coclosed \gtstr s we can use a
micro\-extension trick to reduce the problem to an application
of \cite[Theorem 10.2.1]{eliashberg02} on $M \times (-\epsilon,\epsilon)$.
(There is no apparent way to apply the same trick to closed \gtstr s, which
seem closer to symplectic structures in this sense.)

One motivation for considering coclosed \gtstr s is that they are the
structures induced on 7-manifolds immersed in 8-manifolds with holonomy
$Spin(7)$. One can attempt to construct $Spin(7)$ metrics on
$M \times (-\epsilon, \epsilon)$ using the `Hitchin flow' of coclosed \gtstr s
\cite{hitchin01}. Bryant \mbox{\cite[Theorem 7]{bryant10}} shows that this can
be solved provided that the initial coclosed \gtstr{} is real analytic.

Theorem \ref{thm:hp} implies that any spin 7-manifold $M$ admits smooth
coclosed \gtstr s. When $M$ is closed, Grigorian \cite{grigorian13} proves
short-time existence of solutions $\varphi_t$ for a version of the `Laplacian
coflow' of coclosed \gtstr s. Even if the initial \gtstr{} $\varphi_0$ is
merely smooth, the coclosed \gtstr s $\varphi_t$ will be real analytic for
$t > 0$ (sufficiently small so that the solution exists). 
As a consequence, we deduce the following

\begin{cor}
For every closed spin 7-manifold $M$, $M \times (-\epsilon, \epsilon)$ admits 
torsion-free $Spin(7)$-structures.
\end{cor}

\subsection{Counting deformation classes of $G_2$-structures} \label{subsec:Intro_Diff}
We can think of the set of deformation-equivalence classes of \gtstr s as the
quotient (isomorphic to $\GTD(M)$) of $\GTH(M)$ under the action
\[ \GTH(M) \times \DiffSpin(M) \to \GTH(M), \quad
([\varphi], f) \mapsto [f^*\varphi] .\]
The deformation invariance of $\nu$ implies that this action on
$\GTH(M) \cong \Z$ is by translation by multiples of 24, so that $\GTD(M)$
has at least 24 elements. To determine to what extent $\nu$ classifies elements
of $\GTD(M)$ we need to understand precisely which multiples of 24 are realised
as translations.
Combining the characteristic class formula \eqref{eq:char} with Lemma
\ref{lem:d_speuler} we arrive at 

\begin{prop} 
 Let $f \colon M \cong M$ be a spin diffeomorphism with mapping torus $T_f$.
Then
\label{prop:Dshift}
 \[ D(\varphi, f^*\varphi) = 24 \ahat(T_f) \in \Z.\]
  %
\end{prop}
%
%

The possible values of $\ahat(T_f)$ are closely related to the spin
characteristic class $p_M : = \frac{p_1}{2}(M)$ (see \S \ref{subsec:p_1/2}).
More precisely, the theory developed in \cite{7class} identifies the 
following two key quantities:
\[ \mdiv(M) :=
\left \{ \begin{array}{cl} 0 & \text{if $p_M$ is torsion,} \\
{\rm Max} \{ s \,|\,s, \mpl \in \Z, \; \mpl^2s \text{~divides~} \mpl p_M \} &
\text{otherwise,} \end{array} \right.  \]
and a certain value $r \in \{0,1,2\}$ that depends on the properties of the
automorphisms of $H^4(M)$ preserving $p_M$ and the torsion linking form.
If $H^4(M)$ is torsion-free then $\mdiv(M)$ is simply the greatest integer
dividing $p_M$, and $r = 1$ whenever $H^4(M)$ is 2-torsion-free.
$\mdiv(M)$ is always even by Lemma \ref{lem:p_1/2}.

%
%

The following theorem gives lower bounds on $|\GTD(M)|$.
For $\frac{a}{b}$ a fraction without common factors, denote
$\NumB{\frac{a}{b}} = a$. 

\begin{thm}  \label{thm:spinauto_LB}
If $p_M = 0 \in H^4(M; \Q)$ then $\GTD(M) \equiv \GTH(M) \equiv \Z$.
In general
\[ |\GTD(M)| \geq 24 \cdot \NumB{\frac{2^r\dinfty{M}}{224}} . \]
\end{thm}

\noindent
So, in particular, if $H^4(M)$ has no 2-torsion then
$|\GTD(M)| \geq 24 \cdot \NumB{\frac{\dinfty{M}}{112}}$.


For upper bounds on $|\GTD(M)|$ we need spin
diffeomorphisms $f \colon M \cong M$ with $D(\varphi, f^*\varphi) \neq 0$.
When $M$ is $2$-connected and $p_M$ is not torsion, these are provided
by \cite{7class}. 

\begin{thm}
\label{thm:spinauto_UB} 
If $M$ is 2-connected and $p_M \not= 0 \in H^4(M; \Q)$ then
\[ |\GTD(M)| = 24 \cdot \NumB{\frac{2^r \mdiv(M)}{224}} ; \]
then also
$|\GTD(N \sharp M)| \leq 24 \cdot \NumB{ \frac{2^r \mdiv(M) }{224} }$
for any connected spin 7-manifold $N$.
\end{thm}

Theorem \ref{thm:spinauto_UB} 
helps identify certain manifolds $M$
for which $\nu$ is a complete invariant of $\GTD(M)$.

\begin{cor}
\label{cor:nuauto}
If $2^r\mdiv(M_0)$ divides $224$ for some 2-connected $M_0$ such that
$M \cong N \sharp M_0$, then $|\GTD(M)| = 24$.
In this case two \gtstr s $\varphi$ and $\varphi'$ on $M$ are
deformation-equivalent if and only if $\nu(\varphi) = \nu(\varphi')$.
\end{cor}

\subsection{The $\xi$-invariant}
\label{subsec:intro_xi}

We now describe a further invariant that, depending on
the topology of~$M$, can distinguish more classes of $\GTD(M)$.
For the moment we restrict to the special case when $p_M$ is rationally
trivial, and postpone the full definition to \S \ref{subsec:xi}.

In dimension 7, the Eells-Kuiper invariant $\mu$ arises from
considering the following characteristic class formula
\cite[\S 6]{eells62}: if $X$ is a closed spin 8-manifold then
\begin{equation}
\label{eq:ek_ahat}
224\ahat(X) = p_X^2 - \sign(X) .
\end{equation}
If $M$ is closed spin with $p_M$ a torsion
class and $W$ is a spin coboundary, then $p_W \in H^4(W;\Q)$ is in the image of the compactly supported cohomology $H^4_{cpt}(W;\Q)$, and $p_W^2 \in \Q$ is
well-defined. Then \eqref{eq:ek_ahat} implies that the $\ahat$ defect,
\begin{equation}
\label{eq:classic_ek}
\mu(M) := \frac{p_W^2 - \sign(W)}{8} \in \Q/28\Z,
\end{equation}
is independent of the choice of $W$. (This differs from the definition in
\cite{eells62} by a factor of 28. The mod $\Z$ residue of $\mu(M)$
is determined by the almost-smooth structure of $M$ because $p_W$ is a
characteristic element for the intersection form; therefore $\mu(M)$ can
take $28$ different values if the underlying almost-smooth
manifold is fixed.)

If we consider a \gtstr{} $\varphi$ on a spin manifold $M$ such that $p_M$
is torsion, then we can in a sense cancel the ambiguities in the definitions
of the $\ahat$ defects $\nu$ and $\mu$ to obtain a stronger invariant.
A linear combination of \eqref{eq:defect} and \eqref{eq:ek_ahat} gives that
\[ 7\chi(X) + \frac{3p_X^2 - 45 \sign(X)}{2} = 0  \]
for any closed $X^8$ with $Spin(7)$-structure. 
Hence setting
\begin{equation}
\label{eq:simplest_xi}
\xi(\varphi) := 7\chi(W) + \frac{3p_W^2 - 45\sign(W)}{2} \in \Q
\end{equation}
is independent of choice of $Spin(7)$-coboundary $W$. If we consider \gtstr s
on a fixed smooth $M$ with $p_M$ torsion then the relation
\[ \xi(\varphi) = 7\nu(\varphi) + 12 \mu(M) \mmod 336\Z \]
means that $\nu(\varphi)$ can be determined from $\xi(\varphi)$ and $\mu(M)$.
The $\xi$-invariant takes precisely the values allowed by the constraint
\[ \xi(\varphi) = 7 \chi_\Q(M) + 12\mu(M) \mmod 14\Z . \]
Similarly to Proposition \ref{prop:d_nu},
$14D(\varphi, \varphi') = \xi(\varphi') - \xi(\varphi)$, 
so $\xi$ distinguishes between all elements of $\GTH(M)$. Since $\xi$ is patently invariant under
diffeomorphisms, this entails the claim from Theorem
\ref{thm:spinauto_UB} that $\GTH(M) = \GTD(M)$ when $p_M$ is torsion.

\begin{ex}
\label{ex:rdnu}
$S^7$ has a standard \gtstr{} $\varphi_{rd}$, induced as the
boundary of $B^8$ with a flat \spstr. Clearly
$\nu(\varphi_{rd}) \equiv \chi(B^8) -3\sign(B^8) \equiv 1$.
Meanwhile $p_{B^8} = 0$, so $\xi(\varphi_{rd}) = 7$.

On the other hand, the flat \spstr{} on the complement of $B^8 \subset \bbr^8$
induces the \gtstr{} $-\varphi_{rd}$ on $S^7$ (with the orientation reversed).
If $r$ is a reflection of $S^7$ then $\rdbar = r^*(-\varphi_{rd})$ is a
different \gtstr{} on $S^7$ inducing the same orientation as $\varphi_{rd}$.
Since $\nu(\rdbar) = \nu(-\varphi_{rd}) = -\nu(\varphi_{rd}) = -1$
(and $\xi(\rdbar) = \xi(-\varphi_{rd}) = -\xi(\varphi_{rd}) = -7$)
there can be no homotopy between $\varphi_{rd}$ and $\rdbar$. 
\end{ex}

\begin{ex}
\label{ex:sq}
$S^7$ has a `squashed' \gtstr{} $\varphi_{sq}$ that is invariant under
$Sp(2)Sp(1)$ and nearly parallel (\ie the corresponding cone metric on $\bbr \times S^7$ has exceptional holonomy $Spin(7)$).
This \gtstr{} is the asymptotic link of the asymptotically
conical $Spin(7)$-manifold constructed by Bryant and Salamon \cite{bryant89}
on the total space $W$ of the positive spinor bundle of $S^4$.
This bundle is $\co(-1)$ over $\bbh P^1$ with the orientation reversed. Since
this space has $\sigma = 1$ and $\chi = 2$, it follows that
$\nu(\varphi_{sq}) = 2 - 3 = -1$. 

Further $p_W^2 = 1$, so $\xi(\varphi_{sq}) = -7$. In particular,
$\varphi_{sq}$ is homotopic to $\rdbar$;
if we glue $W$ and $B^8$ to form $\bbh P^2$ then we can
interpolate to define a \spstr{} on $\bbh P^2$.
\end{ex}

The definition of $\xi$ becomes more involved when $p_M$ is rationally
non-trivial.
In general, let $d_\pi$ denote the greatest integer dividing
$p_M$ modulo torsion (which is even by Lemma \ref{lem:p_1/2}),
and $\td_\pi := \gcd(d_\pi, 4)$.
One can then replace the $p_W^2 \in \Q$ that appears in
\eqref{eq:simplest_xi} with a $\Q/2\td_\pi\Z$-valued function on
$\divsp := \{k \in H^4(M) : p_M - d_\pi k \textrm{ is torsion}\}$.
Hence one can define $\xi(\varphi)$ as a function ${\divsp \to \Q/3\td_\pi\Z}$,
see Definition \ref{def:xi}.
It is invariant in the sense that if $f : M' \to M$ is a diffeo\-morphism, then
$f^* : H^4(M) \to H^4(M')$ restricts to a bijection $\divsp \to \divsp'$,
and $\xi(f^*\varphi) \circ f^* = \xi(\varphi)$ for any \gtstr{} $\varphi$
on $M$.

\begin{lem}
\begin{equation}
\label{eq:d_xi}
\xi(\varphi') - \xi(\varphi) = 14D(\varphi, \varphi') \mod 3\td_\pi
\end{equation}
\end{lem}

Together with Proposition \ref{prop:d_nu}, this means that the values
of $(\nu,\xi)$ determine $D(\varphi, \varphi')$ modulo
$\lcm(24, \Num\big(\frac{3\td_\pi}{14}\big)) = 24\NumB{\frac{d_\pi}{112}}$.
However, this does not mean that the pair $(\nu,\xi)$ distinguishes between
$24\NumB{\frac{d_\pi}{112}}$ classes in $\GTD(M)$, but only that it
distinguishes that many classes modulo homotopies and diffeomorphisms
\emph{acting trivially on cohomology}.
The reason is that for a general diffeomorphism $f$ of $M$,
$\xi(\varphi) \circ f^* - \xi(\varphi)$ can be a non-zero
constant in $\Q/3\td_\pi\Z$. Understanding the action of $f$ on $\xi$ reduces
to the same technical problem as for the action on $\GTH(M)$, and we find that
in general $(\nu,\xi)$ can distinguish between 
$24\NumB{\frac{2^r\dinfty{M}}{224}}$ elements of $\GTD(M)$,
which in a sense is a more precise version of Theorem \ref{thm:spinauto_LB}.
In particular, combining with Theorem \ref{thm:spinauto_UB} we find

\begin{thm}
\label{thm:nuxi}
If $M$ is 2-connected then $(\nu,\xi)$ is a complete invariant of $\GTD(M)$.
\end{thm}

In combination with the diffeomorphism classification of closed 2-connected
7-manifolds from \cite{7class}, we obtain a classification result for
2-connected 7-manifolds with \gtstr s, stated in Theorem \ref{thm:classify_g2}.

\subsection{Further problems}
\label{subsec:further}

The main motivation for this work is to help distinguish between connected
components of the moduli space of \gtmetric s on a fixed $M$. One supply of
candidates comes from 2-connected twisted connected sums, but Theorem
\ref{thm:nu24} shows that $\nu$ is not enough to distinguish between those.
All twisted connected sum \gtmfd s $M$ have $\mdiv(M)$ a divisor of
$\mdiv(K3) = 24$, so when $M$ is 2-connected, the only remaining chance of
using the homotopy theory to distinguish between different twisted connected
sums \gtmetric s is when $\mdiv$ is divisible by 3:
by Theorem \ref{thm:spinauto_LB} there are in this case 3 different homotopy
classes of \gtstr s with $\nu = 24$, and they can be distinguished by $\xi$.
A number of examples with $\mdiv(M) = d_\pi(M) = 6$ are exhibited
in \cite{g2m}, and it seems likely that a more exhaustive search will provide
diffeomorphic pairs of such twisted connected sums, but we do not currently
have any way to compute $\xi$ in this situation. 

The examples of Joyce with odd $\nu$ mentioned above can be viewed as a kind of
twisted connected sums, gluing asymptotically cylindrical manifolds with
holonomy a proper subgroup of $G_2$ and cross-section $K3 \times T^2$, but
where the torus factor is not rectangular (as for usual twisted connected sums)
but hexagonal. Such ``extra-twisted connected sums'' provide candidates of
\mbox{2-connected} \gtmfd s with fewer restrictions on the possible values of
$\nu$, and we will return to this elsewhere.

The definition of $\nu$ in terms of a coboundary is not always amenable to 
explicit computations. 
A common theme in differential topology is to find ways to express
`extrinsic' invariants (defined in terms of a coboundary) intrinsically,
\eg the classical Eells-Kuiper invariant can be expressed in terms of
eta invariants \cite{donnelly75}.
Sebastian Goette informs us that it is possible to express $\nu$ analytically,
and we plan to study this and applications to extra-twisted connected sums
further in future work.

Some necessary conditions are known for a closed spin 7-manifold $M$ to admit
a metric with holonomy $G_2$ (see \eg \cite[\S10.2]{joyce00}), but there is
currently no conjecture as to what the right sufficient conditions would be.
A refinement of this already very hard problem would be to ask: which
deformation classes of \gtstr s on $M$ contain torsion-free \gtstr s?
This is of course related to the problem of whether there is any $M$
with torsion-free \gtstr s that are not deformation-equivalent, which was one
of our motivations for introducing $\nu$.
If one attempts to find torsion-free \gtstr s as limits of a flow of \gtstr s
as in \cite{bryant11, grigorian13, weiss12, xu09}, does the homotopy class of
the initial \gtstr s affect the long-term behaviour of the flow?

\subsection*{Organisation}
The rest of the paper is organised as follows.  In Section
\ref{sec:preliminaries} we establish preliminary results needed to define and
compute $\nu$.
In Section \ref{sec:invariant} we define the affine difference $D(\varphi,
\varphi')$ and the $\nu$-invariant, establish the existence of
$Spin(7)$-coboundaries for \gtstr s and hence prove Theorem \ref{thm:defnu}.
We also describe examples of \gtstr s on $S^7$ in more detail.
In Section \ref{sec:torsionfreeG2} we compute the $\nu$-invariant for
twisted connected sum \gtmfd s, proving Theorem \ref{thm:nu24}.
Section \ref{sec:hp} establishes the h-principle for coclosed \gtstr s stated
in Theorem \ref{thm:hp}.
In Section~\ref{sec:diffeomorphisms} we describe the action of spin
diffeomorphisms on $\GTH(M)$, give the general definition of the
\mbox{$\xi$-invariant} and prove
the results from~\S \ref{subsec:Intro_Diff}-\ref{subsec:intro_xi}.

\subsection*{Acknowledgements}
JN thanks the Hausdorff Institute for Mathematics for support and excellent
working conditions during a visit in autumn 2011, from which this project
originates. JN acknowledges post-doctoral support from ERC Grant 247331.
DC thanks the Mathematics Department at Imperial College for its hospitality 
and acknowledges support from EPSRC Mathematics Platform grant EP/I019111/1.
DC also acknowledges the support of the Leibniz Prize of Wolfgang L\"{u}ck, granted by the Deutsche Forschungsgemeinschaft.

\section{Preliminaries} \label{sec:preliminaries}
In this section we describe \gtstr s and $Spin(7)$-structures on $7$ and
$8$-manifolds, and their relationships to spinors.
We also establish some basic facts about the characteristic classes
of spin manifolds in dimensions $7$ and $8$.

\subsection{The Lie groups $Spin(7)$ and $G_2$} \label{subsec:strs}
We give a brief review of how $Spin(7)$ and \gtstr s can be
characterised in terms of forms. For more detail on the differential geometry
of such structures, and how they can be used in the study metrics with
exceptional holonomy, see \eg Salamon \cite{salamon89} or Joyce \cite{joyce00}.
We defer the analogous discussion of $SU(3)$ and $SU(2)$-structures until
we use it in \S \ref{sec:torsionfreeG2}.

The stabiliser in $\gl8$ of the $4$-form
\begin{equation}
\label{sp7formeq}
\begin{aligned}
\psi_{0} = \; & dx^{1234} + dx^{1256} + dx^{1278} + dx^{1357} - dx^{1368}
-dx^{1458} - dx^{1467} - \\ &  dx^{2358} - dx^{2367} - dx^{2457} + dx^{2468}
+ dx^{3456} + dx^{3478} + dx^{5678} \; \in \Lambda^{4}(\bbr^{8})^{*}
\end{aligned}
\end{equation}
is $Spin(7)$ (identified with a subgroup of $SO(8)$ by the spin
representation). Here and elsewhere, $dx^{1234}$ abbreviates
$dx^1 \wedge dx^2 \wedge dx^3 \wedge dx^4$ etc.
On an 8-dimensional manifold $X$, 
a $4$-form $\psi \in \Omega^4(X)$ which is 
pointwise equivalent to $\psi_0$ defines a $Spin(7)$-structure, and induces
a metric and orientation (the orientation form is $\psi^2$).

The exceptional Lie group $G_{2}$ can be defined as the automorphism group
of $\bbo$, the normed division algebra of octonions. Equivalently, $G_{2}$ is
the stabiliser in $\gl7$ of the $3$-form
\begin{equation}
\label{g2formeq}
\varphi_{0} = dx^{123} + dx^{145} + dx^{167} + dx^{246}
 - dx^{257} - dx^{347} - dx^{356} \in \Lambda^{3}(\bbr^{7})^{*} .
\end{equation}
On a 7-dimensional manifold $M$, a $3$-form $\varphi \in \Omega^3(M)$ that is
pointwise equivalent to $\varphi_0$ defines a \gtstr, which induces a
Riemannian metric and orientation. Note that
\begin{equation}
\label{eq:product_sp}
dt \wedge \varphi_0 + *\varphi_0 \cong \psi_0
\end{equation}
on $\bbr \oplus \bbr^7$, so the stabiliser in $Spin(7)$ of a non-zero vector
in $\bbr^8$ is exactly $G_2$. Therefore the product of a 7-manifold with a
\gtstr{} and $S^1$ or $\bbr$ has a natural product $Spin(7)$-structure,
while a $Spin(7)$-structure $\psi$ on $W^8$ induces a \gtstr{} on $\del W$
by contracting $\psi$ with an outward pointing normal vector field.

\begin{rmk}
If $\varphi$ is \gtstr{} on $M^7$, then $-\varphi$ is a \gtstr{} too, inducing
the same metric and opposite orientation (because $\varphi_0$ is equivalent to
$-\varphi_0$ under the orientation-reversing iso\-morphism $-1 \in O(7)$).
The product $Spin(7)$-structure $dt \wedge \varphi + *\varphi$ on
$M \times [0,1]$ induces $\varphi$ on the boundary component
$M \times \{1\} \cong M$, and $-\varphi$ on $M \times \{0\} \cong -M$.
\end{rmk}

\subsection{$G_2$-structures and spinors} \label{subsec:G2spinors}

In this paper we are concerned with \gtstr s on a manifold $M^7$ up to
homotopy. Since there is an obvious way to reverse the orientation of a \gtstr,
while any two Riemannian metrics are homotopic, we may as well consider
\gtstr s compatible with a fixed orientation and metric.
Because $G_2$ is simply-connected,
the inclusion $G_2 \into SO(7)$ lifts to $G_2 \into Spin(7)$.
Therefore a \gtstr{} on $M$ also induces a spin structure, and we focus on
studying \gtstr s compatible also with a fixed spin structure. As in
the introduction, we let $\GTH(M)$ denote the homotopy classes of \gtstr s
on $M$ with a choice of spin structure.

As we already saw, $G_2$ is exactly the stabiliser of a non-zero vector in the
spin representation $\spinr$ of $Spin(7)$; as a representation of $G_2$, $\spinr$ splits
as the sum of a 1-dimensional trivial part and the standard 7-dimensional
representation. $Spin(7)$ acts transitively on the unit sphere in $\spinr$ with
stabiliser $G_2$, so $Spin(7)/G_2 \cong S^7$.

From the above, we deduce that given a spin structure on $M$, a compatible
\gtstr{} $\varphi$ induces an isomorphism $\spinb M \cong \trivr \oplus TM$ for
the spinor bundle $\spinb M$:  here $\trivr$ denotes the trivial line bundle.
Hence we can associate to $\varphi$ a unit section of $\spinb M$, well-defined
up to sign.
Conversely, any unit section of $\spinb M$ defines a compatible \gtstr.
A transverse section $\spins$ of the spinor bundle $\spinb M$ of a spin 7-manifold
has no zeros, so defines a \gtstr; thus a 7-manifold admits \gtstr s if and
only if it is spin (\cf Gray \cite{gray69}, Lawson--Michelsohn
\cite[Theorem IV.10.6]{lawson89}).

Note that $\spins$ and $-\spins$ are always homotopic, because they correspond to
sections of the trivial part in a splitting
$\spinb M \cong \trivr \oplus TM$ and the Euler class of an oriented
7-manifold vanishes.   It follows that
$\spinb M$ contains a trivial $2$-plane field $K \supset \trivr$
which accommodates a homotopy from $\spins$ to $-\spins$.
Therefore $\GTH(M)$ can be identified with homotopy
classes of unit sections of the spinor bundle. As stated in the introduction,
Lemma \ref{lem:GTH=Z} now
follows by a standard application of obstruction theory, but we will describe
the bijection $\GTH(M) \cong \bbz$ in elementary terms in \S \ref{subsec:D}.

\begin{rmk}
Let us make some further comments on the signs of the spinors.
Given a principal $Spin(7)$ lift $\tf$ of the frame bundle $F$ of $M$,
the principal $G_2$-subbundles of $\tf$ are in bijective 
correspondence with sections of the associated unit spinor bundle. The
$G_2$-subbundles corresponding to spinors $\spins$ and $-\spins$ have the same
image in $F$, hence
they define the same \gtstr{} on $M$ (they have the same 3-form $\varphi$).

While $SO(7)$ does not itself act on $\spinr$, the action of $Spin(7)$ on
$(\spinr-\{0\} ) /\bbr^* \cong \bbr P^7$ does descend to an action of $SO(7)$. Therefore the
orbit $SO(7)\varphi_0$, the set of \gtstr s on $\bbr^7$ defining the same
orientation and metric as $\varphi_0$, is $SO(7)/G_2 \cong \bbr P^7$.
\gtstr s compatible with a fixed orientation and metric on $M$ but without any
constraint on the spin structure therefore correspond to sections of
an $\bbr P^7$ bundle. If $M$ is not spin then this bundle has no sections.
Given a spin structure, the unit sphere bundle in the associated spinor bundle
is an $S^7$ lift of the $\bbr P^7$-bundle, and two \gtstr s induce the same spin
structure if they can both be lifted to the same $S^7$ bundle.
\end{rmk}

\subsection{$Spin(7)$-structures and characteristic classes of
$Spin(8)$-bundles}
\label{subsec:spin78}

The spin representation of $Spin(7)$ is faithful, so defines an
inclusion homo\-morphism $Spin(7) \into SO(8)$, which has
a lift $i_\spinr : Spin(7) \into Spin(8)$.
The restriction of the positive half-spin representation $\spinrp$ of
$Spin(8)$ to $Spin(7)$ is a sum of a trivial rank 1 part and the 7-dimensional
vector representation (factoring through $Spin(7) \to SO(7)$). Therefore
$i_\spinr(Spin(7)) \subset Spin(8)$ can be characterised as the stabiliser of a
unit positive spinor $s_0 \in \spinrp$, and $Spin(7)$-structures on a spin
8-manifold are equivalent
to unit positive spinor fields (up to sign, in the same sense as \gtstr s).
Hence there is an obvious obstruction to the existence of $Spin(7)$-structures
on an 8-manifold $X$: it must be spin, and the Euler class in $H^8(X)$ of the
positive half-spinor bundle on $X$ must vanish.

\begin{rmk}
\label{rmk:restrictions}
One can of course also define an embedding $i_0 : Spin(7) \into Spin(8)$ as
the stabiliser of the coordinate vector $e_8$ in the vector representation
$\bbr^8$ of $Spin(8)$. 
The restrictions to this copy of $Spin(7)$ of the half-spin representation
$\spinrpm$ of $Spin(8)$ are both isomorphic to the spin representation of
$Spin(7)$. Therefore, if $W^8$ is a spin manifold then the restrictions
of the half-spinor bundles $\spinbpm W$ to $\partial W$ are naturally
isomorphic to the spinor bundle $\spinb(\partial W)$.

In particular, a positive spinor field on $W^8$ can be restricted to a spinor
field on $\partial W$, so the restriction of a
$Spin(7)$-structure on $W$ to a \gtstr{} on $\partial W$ can be described
in terms of the spinorial picture. Of course, this gives exactly the same
result as if we describe the restriction in terms of differential forms.
This is because the image of the composition of the inclusions
$G_2 \into Spin(7) \stackrel{i_0}{\into} Spin(8)$ is equally well described
as the stabiliser in $Spin(8)$ of $(s_0, e^8) \in \spinrp \times \bbr^8$ and
as the lift of the stabiliser in $GL(\bbr, 8)$ of
$(\psi_0, e_8) \in \Lambda^4 \bbr^8 \times \bbr^8$.
\end{rmk}


Let us describe briefly our conventions for orientations on the half-spin
representations of $Spin(8)$. For each fixed non-zero $v \in \bbr^8$, the
Clifford multiplication $\bbr^8 \times \spinrpm \to \spinrmp$ defines
orientation-preserving isomorphisms $c_v^\pm : \spinrpm \to \spinrmp$.
A feature of the `triality' in dimension 8 is that the map
$\rcliff_{\spins_\pm} : \bbr^8 \to \spinrmp$ induced by
Clifford multiplication with a fixed non-zero spinor $\spins_\pm \in \spinrpm$
is an isomorphism too. The Clifford relations imply that,
for $\spins_+ = v\spins_-$,
\[ c_v^+ \circ \rcliff_{\spins_-} = \; \rcliff_{\spins_+} \circ r_v \; :
\; \bbr^8 \to \spinrn, \]
where $r_v: \bbr^8 \to \bbr^8$ is reflection in the 
hyperplane orthogonal to $v$.
Thus $\rcliff_{\spins_+}$ and $\rcliff_{\spins_-}$  have opposite
orientability. Our convention is that $\rcliff_{\spins_-}$ is
orientation-preserving, while $\rcliff_{\spins_+}$ is not.

More explicitly, $\bbr^8$, $\spinrp$ and $\spinrn$ can each be identified
with the octonions $\bbo$ so 
that the Clifford multiplication $\bbr^8 \times \spinrn \to \spinrp$
corresponds to the octonionic multiplication $(x,y) \mapsto xy$, \cf Baez
\cite[p.162 above (5)]{baez02}. Then, to
satisfy the Clifford relations, $\bbr^8 \times \spinrp \to \spinrn$ must
correspond to $(x,y) \mapsto -\bar x y$, where $\bar x$ is the octonion
conjugate of $x$. This map is orientation-reversing on the first factor.


Let $X$ be a spin 8-manifold, $e \in H^8(X)$ the Euler class of $TX$, and
$e_\pm \in H^8(X)$ the Euler classes of the half-spinor bundles $\spinbpm X$.
More generally, for any principal $Spin(8)$-bundle on any~$X$, let $e, e_\pm$
denote the Euler classes of the vector bundles associated to the vector
and half-spin representations of $Spin(8)$.
With our orientation conventions, the non-degeneracy of the Clifford
product implies
\begin{equation}
\label{eq:cliff}
e_+ = e + e_- .
\end{equation} 

The following statement can be found for instance in
Gray--Green \cite[p.89]{gray70}. 
\begin{prop}\label{prop:e+}
For any principal $Spin(8)$-bundle
\[ e_\pm = \frac{1}{16}\left(p_1^2 - 4p_2 \pm 8e \right). \]
\end{prop}
In degree 8, the $\ahat$ and $L$ genera are given by
\begin{equation}
\label{eq:genera}
\begin{aligned}
45 \cdot 2^7 \ahat &= 7 p_1^2 - 4p_2, \\
45 L &= 7p_2 - p_1^2,
\end{aligned}
\end{equation}
so Proposition \ref{prop:e+} can be rewritten as
$e_\pm = 24 \ahat + {\displaystyle\frac{\pm e - 3L}{2}}$.
If $X$ is closed and orientable then the integral of the $L$ genus of $TX$ is
the signature of $X$ by the Hirzebruch signature theorem, while the integral of
the Euler class is just the ordinary Euler characteristic.

\begin{cor}
\label{cor:ahat}
If $X$ is a closed spin 8-manifold then
\[ e_\pm(X) = 24 \ahat(X) + \frac{\pm\chi(X) - 3\sign(X)}{2} . \]
\end{cor}


\begin{rmk}
Modulo torsion, the group of 
integral characteristic classes of a principal $Spin(8)$-bundle in dimension
$8$ is generated by $p_1^2$, $p_2$ and $e$, so we could prove
Corollary \ref{cor:ahat} (and hence Proposition \ref{prop:e+}) by checking
that the formula holds for the following spin $8$-manifolds.
\begin{itemize}
\item $S^8$: $\chi = 2$, $\ahat = \sign = 0$, $e_\pm = \pm 1$.
\item $K3 \times K3$: $\chi = 24^2$, $\sign = (-16)^2$. $\ahat = 4$ because
the holonomy is $SU(2) \times SU(2)$. Because this also defines a \spstr{}
(\cf \eqref{eq:su2^2}), $e_+ = 0$ and $e_- = -\chi$.
\item $\bbh P^2$: $\chi = 3$, $\sign = 1$. $\ahat = 0$ by the Lichnerowicz
formula since there is a metric with positive scalar curvature.
$e_- = -\chi$ because $\spinbn X \cong -TX$ for any spin 8-manifold $X$ with
$Sp(2)Sp(1)$-structure. This structure also splits $\spinbp X$ into a sum of a
rank 5 and a rank 3 part,
so $e_+ = 0$. (Alternatively, we can identify a quaternionic line subbundle of
$T\bbh P^2$, like that spanned by the projection of the vector field
$(q_1, q_2, q_3) \mapsto (0, q_1, q_2)$ on $\bbh^3$, with a non-vanishing
section of the rank 5 part of $\spinbp X$.)
\end{itemize}
\end{rmk}

\section{The $\nu$-invariant} \label{sec:invariant}

In this section we study the set $\GTH(M)$ of homotopy classes of \gtstr s
on a closed spin 7-manifold $M$, and prove the basic properties of the
invariants $D$ and $\nu$. We conclude the section with some concrete examples.

\subsection{The affine difference} \label{subsec:D}

Let $M$ be a closed connected spin 7-manifold, and $\varphi, \varphi'$ a pair
of \gtstr s on $M$. We describe how to define the difference
$D(\varphi, \varphi') \in \bbz$ from Lemma \ref{lem:defd}.

A homotopy of \gtstr s is equivalent to a path of non-vanishing spinor fields.
Any path of spinor fields on $M$ can be identified with a positive spinor
field $\spins$ on $M \times [0,1]$. We can always find $\spins$ 
such that the restrictions to $M \times \{1\}$ and $M \times \{0\}$ are
the non-vanishing spinor fields corresponding to $\varphi$ and $-\varphi'$,
respectively. Then the pull-back by $\spins$ of the Thom class of the positive
spinor bundle defines a relative Euler class in $H^8(W,M)$, independent of 
the choice of $\spins$, and we define $D(\varphi,\varphi')$ to be its
integral $n_+(M \times [0,1], \varphi, \varphi')$.
If we take $\spins$ to have transverse zeroes then we can interpret this
geometrically as the intersection number of the graph of $\spins$ with the zero
section.

It is obvious from this definition that the affine relation \eqref{eq:d_affine}
holds.
If $n_+(M \times [0,1], \varphi, \varphi') = 0$ then $\spins$ can be chosen to
be non-vanishing, so $\varphi$ and $\varphi'$ are homotopic if and only if
$D(\varphi, \varphi') = 0$.
Given $\varphi$ we can construct $\varphi'$ such that
$D(\varphi, \varphi') = 1$ by modifying the defining spinor of $\varphi$ in a
7-disc $B^7$: in a local trivialisation we change it from a constant map
$B^7 \to S^7$ to a degree 1 map. Thus $D$ can take any integer value,
so $D$ really corresponds to the difference function under a bijection
$\bbz \cong \GTH(M)$, completing the proof of Lemma \ref{lem:defd}.

To compute $D(\varphi, \varphi')$, we can consider more general spin
8-manifolds $W$ with boundary $M \sqcup {-M}$. Generalising the above, let
$n_+(W, \varphi, \varphi')$ be the intersection number with the zero section of
a positive spinor whose restriction to the two boundary components correspond
to $\varphi$ and $-\varphi'$. Form a closed spin 8-manifold $\ow$ by gluing the
$M$ piece of the boundary of $W$ to the $-M$ piece. We can define a continuous
positive spinor field on $\ow$ by modifying the spinor field from $W$ in a
$M \times [0,1]$ neighbourhood of the former boundary, to interpolate between
$\varphi'$ on $M \times \{1\}$ and $-\varphi$ on $M \times \{0\}$. Its
intersection number with the zero section is 
$n_+(W, \varphi, \varphi') - D(\varphi, \varphi')$, so we can compute $D$ as
\begin{equation}
\label{eq:dbord}
D(\varphi, \varphi') = n_+(W, \varphi, \varphi') - e_+(\ow) .
\end{equation}

\subsection{The definition of $\nu$} \label{subsec:definition_of_nu}

Let $M$ be a closed spin 7-manifold (not necessarily connected)
with \gtstr{} $\varphi$, and $W$ a compact spin 8-manifold with
$\partial W = M$. Such $W$ always exist since the bordism group
$\Omega^{Spin}_7$ is trivial \cite{milnor63}. The restrictions of the
half-spinor bundles $\spinbpm W$ of $W$ to $M$ are isomorphic to the spinor
bundle on $M$ (Remark \ref{rmk:restrictions}), and
the composition $\spinbp W_{|M} \to \spinbn W_{|M}$ of these isomorphisms is
Clifford multiplication by a unit normal vector field to the boundary.
Let $n_\pm(W,\varphi)$ be the intersection number with the zero section of a
section of $\spinbpm W$ whose restriction to $M$ is the non-vanishing spinor
field defining $\varphi$. Let
\begin{equation}
\label{eq:nubar}
\nubar(W,\varphi) := -2 n_+(W,\varphi) + \chi(W) - 3 \sign(W) \in \bbz .
\end{equation}
Reversing the orientations, $-W$ is a spin 8-manifold whose boundary $-M$
is equipped with a \gtstr{} $-\varphi$.

\begin{lem}
\label{lem:defnu}
Let $W$ be a compact spin 8-manifold, and $\varphi$ a \gtstr{} on $M = \del W$.
\begin{enumerate}
\item \label{it:affine}
If $\varphi'$ is another \gtstr{} on $M$ then
$\nubar(W,\varphi') - \nubar(W, \varphi) = 2D(\varphi,\varphi')$
\item \label{it:Qsemi-char}$\nubar(W,\varphi) \equiv \chi_\Q(M) \mod 2$
\item \label{it:nureverse} $\nubar(-W,-\varphi) = - \nubar(W,\varphi)$
\item \label{it:welldef} 
If $W'$ is another compact spin 8-manifold with $\partial W' = M$ then
the closed spin 8-manifold $X = W \cup_{{\Id}_M} (-W')$ has
\[ 48\ahat(X) = \nubar(W', \varphi) - \nubar(W,\varphi) . \]
\end{enumerate}
\end{lem}

\begin{proof}
\ref{it:affine}
Clearly
$n_+(W,\varphi) = n_+(M \times I, \varphi, \varphi') + n_+(W, \varphi')$. 

\ref{it:Qsemi-char}
For $W^{4n}$ any compact oriented manifold with boundary, $\sign(W)$ is
by definition the signature of a non-degenerate symmetric form
on the image $H^{2n}_0(W)$ of $H^{2n}(W,M) \to H^{2n}(W)$. In particular,
$\sign(W) \equiv \dim H^{2n}_0(W) \mmod 2$. Writing
$\chi(W) = \sum_{i=0}^{2n-1} b_i(W) + \sum_{i=0}^{2n} b_{4n-i}(W)$ and using
$b_{4n-i}(W) = b_i(W,M)$ and
the definition that $\chi_\Q(W) = \sum_{i= 0}^{2n-1} b_i(\del W) \mmod 2$, the
exactness of the sequence
$$0 \to H^0(W,M) \to H^0(W) \to \cdots \to H^{2n-1}(\del W) \to H^{2n}(W,M)
\to H^{2n}_0(W) \to 0$$ implies 
\begin{equation}
\label{eq:sign_chi_mod2}
\sign(W) +\chi(W) \equiv 
\chi_\Q(\del W) \mod 2 .
\end{equation}

\ref{it:nureverse}
Let $v$ be a vector field on $W$ that is a unit outward-pointing normal field
along $M$, and $\spins \in \Gamma(\spinbp W)$ a spinor field whose restriction
to $M$ 
induces $\varphi$. Then the
restriction of the Clifford product $v \cdot \spins \in \Gamma(\spinbn W)$ also
induces $\varphi$. By the
Poincare--Hopf index theorem, the number of zeros of $v$ is $\chi(W)$, so
$n_-(W,\varphi) = n_+(W,\varphi) - \chi(W)$ (these signs are compatible
with~\eqref{eq:cliff}).

Reversing the orientations swaps sections of $\spinbp W$ and $\spinbn W$, and reverses
the signs assigned to the zeros, so $n_+(-W,-\varphi) = - n_-(W,\varphi)$.
It also reverses the signature, but preserves the Euler characteristic.
Thus
\[ \nubar(-W, -\varphi) = 2n_-(W,\varphi) + \chi(W) + 3 \sign(W) =
2n_+(W,\varphi) - 2\chi(W) + \chi(W) + 3 \sign(W) = -\nubar(W,\varphi) .\]

\ref{it:welldef}
$\sign(W) + \sign(-W') = \sign(X)$ by Novikov additivity \cite[7.1]{atiyah68},
$\chi(W) + \chi(-W') = \chi(X)$ because ${\chi(M) = 0}$, and
$X$ has a transverse positive spinor field whose intersection number with the
zero section is $n_+(W,\varphi) + n_+(-W',-\varphi)$.
Hence 
\[ \nubar(W', \varphi) - \nubar(W, \varphi) =
\nubar(W', \varphi) + \nubar(-W, -\varphi) =
2e_+(X) - \chi(X) + 3\sign(X) = 48\ahat(X) \]
by Corollary \ref{cor:ahat}.
\end{proof}

\begin{cor}
\label{cor:defnu}
$\nu(\varphi) : = \nubar(W,\varphi)~\textup{mod}~48 \in \bbz_{48}$ is independent of the choice
of $W$, and
\[ \nu(\varphi') - \nu(\varphi) \equiv 2D(\varphi,\varphi') \mod 48 . \]
\end{cor}

\noindent
This gives the majority of the proofs of Theorem \ref{thm:defnu} and Proposition
\ref{prop:d_nu}. To complete the proofs it remains only to show the existence
of $Spin(7)$-coboundaries, since Definition \ref{def:nu} 
is phrased in terms of those.  We show the existence of the
required $Spin(7)$-coboundaries in the following subsection.

\subsection{$Spin(7)$-bordisms} \label{subsec:spin7bordisms}

Let $\varphi$, $\varphi'$ be \gtstr s on closed 7-manifolds $M$, $M'$.
A $Spin(7)$-bordism from $(M,\varphi)$ to $(M',\varphi')$ is a compact
8-manifold with boundary $M \sqcup -M'$ and a \spstr{} $\psi$ inducing the
respective \gtstr s on the boundary. More formally, we require that
$\partial W = f(M) \; \sqcup \; f'(M')$ for embeddings $f : M \into \partial W$,
$f' : M' \into \partial W$ that pull back the contraction of $\psi$ with the
outward normal field to $\varphi$ and $-\varphi'$, respectively.
If $M = M'$ then we can form a closed spin 8-manifold by identifying the
boundary components,
\begin{equation}
\label{eq:close}
\ow := W/(f' \circ f^{-1}) .
\end{equation}
Clearly, there is a topologically trivial $Spin(7)$-bordism $W$ (\ie there
is a diffeomorphism $W \cong M \times [0,1]$, but it does not have to preserve
the \spstr) from $\varphi$ to $\varphi'$ if and only if they are
deformation-equivalent, \ie $f^*\varphi'$ is homotopic to $\varphi$ for some
diffeomorphism $f \colon M \cong M$.

\begin{rmk}
If $(W, \psi, f, f')$ is a $Spin(7)$-bordism from $(M,\varphi)$ to
$(M',\varphi')$ then $(W,\psi, f', f)$ is $Spin(7)$-bordism from
$(-M',-\varphi')$ to $(-M, -\varphi)$.
However, it does not follow in general that $-W$ has a $Spin(7)$-structure
making it a $Spin(7)$-bordism from $(M', \varphi')$ to $(M, \varphi)$ (because
the orientation of a $Spin(7)$-structure cannot be reversed).
In particular, if $W$ is a $Spin(7)$-coboundary for $(M,\varphi)$ then
$-W$ is not necessarily a $Spin(7)$-coboundary for $(-M,-\varphi)$,
unless $\chi(W) = 0$, \cf proof of Lemma \ref{lem:defnu}\ref{it:nureverse}.
\end{rmk}

The $Spin(7)$-structure $\psi$ induces a non-vanishing positive spinor field
$\spins$ on $W$. By Remark \ref{rmk:restrictions} the restriction of $s$ to
$\partial W$ is the spinor defining the \gtstr s $\varphi$ and $-\varphi'$,
so $n_+(W, \varphi, \varphi') = 0$.
In particular, when $\varphi$ and $\varphi'$
are \gtstr s on the same manifold $M = M'$, Lemma \ref{lem:d_speuler}
follows from~\eqref{eq:dbord}. Similarly, if $W$ is a $Spin(7)$-coboundary
for $(M, \varphi)$ then $\nubar(W,\varphi) = \chi(W) -3\sign(W) $, so
Corollary \ref{cor:defnu} together with Lemma \ref{lem:spinbordism}(ii) imply
Theorem \ref{thm:defnu}.

\begin{lem}\hfill
\label{lem:spinbordism}
\begin{enumerate}
\item
For a connected compact spin 8-manifold $W$ with connected boundary $M$,
there is a unique homotopy class of \gtstr s on $M$ that bound \spstr s
on $W$.
\item \label{it:spin7bord}
Any \gtstr{} has a $Spin(7)$ coboundary
(any two \gtstr s are $Spin(7)$-bordant).
\end{enumerate}
\end{lem}

\begin{proof}
If $W$ is connected with non-empty boundary then there is no obstruction to
defining a non-vanishing positive spinor field on $W$, so there is some
\gtstr{} $\varphi$ on $M$ that bounds a \spstr{} on $W$.
If $\varphi'$ is another \gtstr{} bounding a \spstr{} on $W$, consider an
arbitrary spin filling $W'$ of $-M$, and let $-\varphi''$ be a \gtstr{} on
$-M$ that bounds a \spstr{} on $W'$. Then $W \sqcup W'$
admits two $Spin(7)$-structures that define bordisms from $\varphi$ and
$\varphi'$, respectively, to $\varphi''$.  Hence 
\[ D(\varphi, \varphi') = D(\varphi, \varphi'') - D(\varphi', \varphi'') = 0 , \]
and $\varphi$ and $\varphi'$ must be homotopic.

For \ref{it:spin7bord}, take any spin filling $W$ of $M$, and let $\varphi$ be
a \gtstr{} on $M$ that bounds a $Spin(7)$-structure. In order to find a
$Spin(7)$-coboundary for some other $\varphi'$ with
$D(\varphi, \varphi') = \pm k$, we use that
if $X$ and $X'$ are closed spin 8-manifolds then, since $\ahat$ and $\sign$
are bordism-invariants, and in particular additive under connected sums,
Corollary \ref{cor:ahat} implies that
\[ e_+(X \sharp X') = e_+(X) + e_+(X') - 1 . \]
(We could also see that for any pair of positive spinor fields $\spins$, $\spins'$ on
$X$, $X'$ one can define a spinor field on $X \sharp X'$ that equals $\spins$ and
$\spins'$ outside the connecting neck, and with a single zero on the neck.)
Therefore $\varphi'$ will bound a \spstr{} on $W'$ the connected sum of $W$
with $k$ copies of a manifold with $e_+ = 2$ or $0$, \eg $S^4 \times S^4$
or $T^8$.
\end{proof}

%

\subsection{Examples of $G_2$-structures on $S^7$}

To make the discussion more concrete, we elaborate on some examples on $S^7$,
%
%
where $D$ can be described in the following direct way.
The spinor bundle of $S^7$ can be trivialised by identifying it with the
restriction of the positive half-spinor bundle on~$B^8$, thus up to homotopy,
a \gtstr{} $\varphi$ on $S^7$ can be identified with a map $f$ from $S^7$ to
the unit sphere in~$\spinrp$.
The difference $D$ between two \gtstr s on $S^7$
equals the difference of the degrees of the corresponding maps $S^7 \to S^7$:
$D(\varphi, \varphi') = \deg f - \deg f'$.

\begin{ex}
\label{ex:round_reverse}
We first illustrate how this description works for the standard round \gtstr{}
$\varphi_{rd}$ and its reverse $\rdbar$, which we already understand from
Example \ref{ex:rdnu}.
By definition, 
$\varphi_{rd}$ corresponds
to a constant map $f_{rd} : x \mapsto \spins_0$. The \gtstr{} $\varphi_{rd}$ is
invariant under the action of $Spin(7)$, and so is $f_{rd}$, in the sense that
$f_{rd}(gx) = \spins_0 = g\spins_0 = gf_{rd}(x)$ for any $g \in Spin(7)$.

Let $r$ be a reflection of $S^7$, and $\rdbar = r^*(-\varphi_{rd})$ as before.
Then $\rdbar$ is invariant under the action of the conjugate subgroup
$rSpin(7)r \subset Spin(8)$. If $x_0 \in S^7$ is a vector orthogonal to the
hyperplane of the reflection, then $\varphi_{rd}$ and $\rdbar$ take the
same value at $x_0$. Thus $\hat f_{rd}(x_0) = \spins_0$, and
$\hat f_{rd}(rgrx_0) = (rgr)\spins_0$ for any $g \in Spin(7)$.
The outer automorphism on $Spin(8)$ of conjugating by $r$ swaps the 
the positive and negative spin representations via Clifford multiplication
by $x_0$, so $(rgr)s_0 = x_0 \cdot(g(x_0 \cdot \spins_0)) =
x_0 \cdot (g(x_0) \cdot \spins_0)$ for $g \in Spin(7)$. Hence
$\hat f_{rd} : S^7 \to S^7$ equals the orientation-preserving diffeomorphism
$c_{x_0}^- \circ \rcliff_{\spins_0} \circ (-r)$, and
$D(\rdbar, \varphi_{rd}) = \deg \hat f_{rd} - \deg f_{rd} = 1$.
\end{ex}


\begin{ex} \label{ex:oct_framing}
Consider the octonionic left-multiplication parallelism on $S^7$, \ie the
trivialisation
of $TS^7$ obtained by considering $u \in S^7$ as a unit octonion and
defining $L_u : \im \bbo \cong T_uS^7$ as left multiplication by $u$.
Its associated \gtstr{} $\varphi_\bbo$ has
$\varphi_\bbo(u) = L_u\varphi_0$ for a fixed \gtstr{} $\varphi_0$.
The associated map $f_\bbo : S^7 \to S^7$ is $u \mapsto \tilde L_u\spins_0$
where $S^7 \to Spin(8), \; u \mapsto \tilde L_u$ is the continuous
lift of ${S^7 \to SO(8),} \; {u \mapsto L_u}$
(with $\tilde L_1 = \Id$) which acts on $\spins_0 \in \spinrp$.

Here is one way to understand $\tilde L_u$. The Moufang identity
$u(xy)u = (ux)(yu)$ holds for any $u,x,y \in \bbo$
\cite[Lemma A.16(c)]{harvey82}, so
$(L_u, R_u, L_u {\circ} R_u) \in SO(8)^3$ preserves the Cayley multiplication.
That can be identified with Clifford
multiplication $\bbr^8 \times \spinrn \to \spinrp$, whose stabiliser in the group
$SO(\bbr^8) \times SO(\spinrn) \times SO(\spinrp)$ is precisely $Spin(8)$
\cite[(5)]{baez02}.
Hence a copy of $S^7$ in $Spin(8)$ whose action on $\bbr^8$ is by $L_u$ must
act on $\spinrp$ by $L_u \circ R_u$. If we choose the identification
$\spinrp \cong \bbo$ so that $\spins_0$ corresponds to 1, then
$f_\bbo(u) = \tilde L_u \spins_0$ corresponds to $u^2$, so $\deg f_\bbo = 2$.
Hence $D(\varphi_\bbo, \varphi_{rd}) = 2$, and $\nu(\varphi_\bbo) = -3$.
\end{ex}

\begin{ex} \label{ex:rd_mod_Z4}
The \gtstr{} $\varphi_{rd}$ is 
invariant under the order 4 diffeomorphism given
by scalar multi\-plication by $i$ on~$S^7 \subset \bbc^4$
(since $i\,\Id \in SU(4) \subset Spin(7)$) so descends to a \gtstr{}
$\varphi_{rd}/\bbz_4$ on the quotient $S^7/\bbz_4$. This is the boundary of the
unit disc bundle of $\co(-4)$ on $\bbc P^3$ (the canonical bundle of
$\bbc P^3$), which has an $SU(4)$-structure restricting to
$\varphi_{rd}/\bbz_4$ (indeed, the total space admits a Calabi--Yau metric
asymptotic to $\bbc^4/\bbz_4$, \cf~Calabi \cite[\S 4]{calabi79}).
The self-intersection number of a hyperplane in the zero-section is $-4$, so
$\sign = -1$, and $\nu(\varphi_{rd}/\bbz_4) = 4 + 3 = 7$.
\end{ex}

\begin{rmk}
While Example \ref{ex:rd_mod_Z4} illustrates that $\nu$ itself is not
multiplicative under covers,
if $\varphi$ and $\varphi'$ are \gtstr s on the same closed spin 7-manifold $M$
and $\cm : \tm \to M$ is a degree $\cd$ covering map then
$D(\cm^*\varphi, \cm^*\varphi') = \cd D(\varphi,\varphi')$.
\end{rmk}

\begin{rmk}
The fact that $\varphi_{rd}$ and $\rdbar$ are both invariant under the
antipodal map on $S^7$ is not incompatible with $D(\varphi_{rd}, \rdbar)$
being odd, because the \gtstr s they define on $\bbr P^7 = S^7/{\pm1}$ induce
different spin structures. The actions of $Spin(7)$ and the conjugate
$rSpin(7)r$ on $\bbr P^7$ can both be lifted to the spinor bundle. Since
$-1$ acts trivially on $\bbr P^7$, its image under either lift will be
$\pm \Id$, and the two spin structures can be distinguished by which of the two
lifts acts as~$+\Id$.

Similarly, $\varphi_{rd}$ defines the same spin structure on $\bbr P^7$ as
the octonionic left-multiplication parallelism of $\bbr P^7$, but not the
right-multiplication one. This is related to the fact that $Spin(7)$ can be
described as the subgroup of $SO(8)$ generated by left multiplication by unit
imaginary octonions, while the subgroup generated by right multiplications is a
conjugate of $Spin(7)$ by a reflection.
\end{rmk}

\section{$\nu$ of twisted connected sum \gtmfdsoft s} \label{sec:torsionfreeG2}

Our motivation for introducing the invariant $\nu$ is to give a tool for
studying the homotopy classes of \gtstr s. 
We now show how the
definition of $\nu$ in terms of $Spin(7)$-bordisms allows us to compute it
for the large class of `twisted connected sum' manifolds with holonomy $G_2$.
Before describing the twisted connected sums, we explain how to compute $\nu$
of \gtstr s defined as products of structures on lower-dimensional manifolds.
This is then used in the proof of Theorem \ref{thm:nu24}, that the torsion-free
\gtstr s of twisted connected sum \gtmfd s always have $\nu = 24$.

\subsection{$SU(3)$ and $SU(2)$-structures}

Let us first describe $SU(3)$ and $SU(2)$-structures in terms of forms,
along the lines of \S \ref{subsec:strs}.

Let $z^k = x^{2k-1} +ix^{2k}$ be complex coordinates on $\bbr^6$. Then 
the stabiliser in $\gl6$ of the pair of forms
\begin{gather*}
\Omega_{0} = dz^1 \wedge dz^2 \wedge dz^3 \in
\Lambda^3(\bbr^6)^{*} \otimes \bbc, \\
\omega_{0} = {\textstyle\frac{i}{2}}(dz^1 \wedge d\bar z^1 +
dz^2 \wedge d\bar z^2 + dz^3 \wedge d\bar z^3) \in \Lambda^2(\bbr^6)^{*},
\end{gather*}
is $SU(3)$. An $SU(3)$-structure $(\Omega, \omega)$ on a 6-manifold induces a
Riemannian metric, almost complex structure and orientation (the volume form
is $-\frac{i}{8}\Omega \wedge \overline \Omega = \frac{1}{6}\omega^3$).
On $\bbr \oplus \bbr^6$
\begin{equation}
dt \wedge \omega_0 + \re \Omega_0 \cong \varphi_0,
\end{equation}
and $SU(3)$ is exactly the stabiliser in $G_2$
of a non-zero vector in $\bbr^7$. The product of a 6-manifold with
$SU(3)$-structure and $S^1$ or $\bbr$ has a product \gtstr, while the boundary
of a 7-manifold with \gtstr{} has an induced $SU(3)$-structure.

The stabiliser in $\gl4$ of the triple of forms
\[ \omega^I_0 = dx^{12} + dx^{34}, \;\omega^J_0 = dx^{13} - dx^{24},
\; \omega^K_0 = dx^{14} + dx^{23} \in \Lambda^2(\bbr^4)^* \]
is $SU(2)$. The stabiliser in $SU(2)$ of a non-zero vector is clearly trivial,
and the boundary of a 4-manifold $W$ with $SU(2)$-structure
$(\omega^I, \omega^J, \omega^K)$ has a natural coframe defined by contracting
each of the three 2-forms with an outward pointing normal vector field.

If $e^1, e^2, e^3$ is a coframe on $\bbr^3$ then
\[ e^{123} + e^1 \wedge \omega^I_0 + e^2 \wedge \omega^J_0 +
e^3 \wedge \omega^K_0 \cong \varphi_0 \]
on $\bbr^3 \oplus \bbr^4$. Therefore the product of a parallelised 3-manifold
and a 4-manifold with $SU(2)$-structure has a natural product \gtstr.
Similarly, if we let $\omega^I_1, \omega^J_1, \omega^K_1$ denote an equivalent
triple of 2-forms on a second copy of $\bbr^4$, and
$\vol_0 = \half(\omega^I_0)^2$ etc, then
\begin{equation}
\label{eq:su2^2}
\vol_0 + \omega^I_0 \wedge \omega^I_1 + \omega^J_0 \wedge \omega^J_1 +
\omega^K_0 \wedge \omega^K_1 + \vol_1 \cong \psi_0
\end{equation}
on $\bbr^4 \oplus \bbr^4$, so the product of two 4-manifolds
$W_0$, $W_1$
with $SU(2)$-structures has a natural product $Spin(7)$-structure.
If $W_0$ is closed while $\partial W_1$ is non-empty, clearly the
\gtstr{} induced on $\partial(W_0 \times W_1)$ by this $Spin(7)$-structure
equals the product of $\omega_0^\bullet$ with the coframe on $\partial W_1$
induced by $\omega_1^\bullet$.

\subsection{Product $G_2$-structures and spinors} \label{subsec:productG2}

Above we described two types of product \gtstr s.
In order to compute $\nu$ of such products, we shall need to describe $SU(3)$ and
$SU(2)$ in terms of spinors.

The half-spin representations $\spinrpm$ of $Spin(6) \cong SU(4)$ are the standard
4-dimensional representation of $SU(4)$ and its dual.  The inclusion
$SU(3) \into SO(6)$ lifts to the obvious inclusion $SU(3) \into SU(4)$, so the
stabiliser of a non-zero element in $\spinrp$ is exactly $SU(3)$. 
Hence, ana\-logously to \S \ref{subsec:G2spinors}, $SU(3)$-structures on a
6-manifold $N$ compatible with a fixed 
spin structure and metric can be defined by positive unit spinor fields
(which always exist and any two are homotopic since the real rank of $\spinbp N$ is 8).

If $N$ is the boundary of a spin 7-manifold $M$, then the half-spinor bundles
on $N$ are both isomorphic, as real vector bundles, to the restriction of
the spinor bundle from $M$. Analogously to Remark \ref{rmk:restrictions},
the restrictions of \gtstr s on $M$ to $SU(3)$-structures on $N$ can be
described equivalently in terms of differential forms or spinors.
As there is no obstruction to extending a
non-vanishing section of a rank 8 bundle on $M$ from the boundary to the
interior, it follows that any $SU(3)$-structure on $N$ is induced as the
boundary of a \gtstr{} on $M$.

\begin{lem}
If $N$ is a 6-manifold with an $SU(3)$-structure
$(\Omega, \omega)$, then the product \gtstr{}
$\varphi = d\theta \wedge \omega + \re \Omega$ on $S^1 \times N$
has $\nu(\varphi) = 0$.
\label{lem:nu_su3}
\end{lem}

\begin{proof}
Any spin 6-manifold $N$ bounds some spin 7-manifold $M$, as the bordism group
$\Omega^{Spin}_6$ is trivial \cite{milnor63}.  Then any product \gtstr{} $\varphi$ on
$S^1 \times N$ bounds a product $Spin(7)$-structure on $S^1 \times M$.
The $S^1$ factor makes $\sign(S^1 \times M) = \chi(S^1 \times M) = 0$,
so $\nu(\varphi) = 0$.
\end{proof}

Now we consider dimensions 3 and 4. Before looking at the spinors we prove
a topological lemma.

\begin{lem}
\label{lem:chimod2}
For any compact spin $4$-manifold $W$ with boundary $\tmfd$,
\[ \chi(W) \equiv \chi_2(\tmfd) \mod 2, \]
where $\chi_2(\tmfd)$ is the mod 2 semi-characteristic
$\sum_{i = 0}^{1} \dim H^i(\tmfd;\bbz_2)$.
\end{lem}

\begin{proof}
Repeating the argument in the proof of \eqref{eq:sign_chi_mod2}  
with $\Z_2$-coefficients instead of $\Q$-coefficients shows that there is a
mod~$2$ identity
\begin{equation*} \label{eq:chi_and_sigma}
	\chi(W) \equiv \dim H^2_0(W; \Z_2) + \chi_{2}(\tmfd) \mod 2,
\end{equation*} 
where $H^{2}_0(W; \Z_2)$ is the image of
$H^{2}(W, \tmfd; \Z_2) \to H^{2}(W; \Z_2)$. The intersection form of $W$
defines a non-singular bilinear form over $\Z_2$ on $H^{2}_0(W; \Z_2)$. 
This injects as an orthogonal summand into the mod 2 intersection form of the
manifold $X : = W \cup_{\Id_\tmfd} -W$.  Since $X$ is a closed
spin $4$-manifold, its intersection form is even, 
and hence the form on $H^{2}_0(W; \Z_2)$ is too.
By \cite[Ch.\,III Lemma 1.1]{milnor73} the rank of every non-singular even bilinear form over $\Z_2$
is even, which completes the proof.
\end{proof}

The spin representations of $Spin(4) \cong SU(2) \times SU(2)$ are the standard
2-dimensional complex representations of the two factors. 
Therefore the stabiliser of a non-zero positive spinor is one of the $SU(2)$
factors, and a unit spinor field on a spin 4-manifold defines an
$SU(2)$-structure.

The spin representation of $Spin(3) \cong SU(2)$ is again the standard
representation of $SU(2)$. 
The stabiliser of a non-zero spinor is trivial,
so a unit spinor field defines a parallelism, \ie a trivialisation of the
tangent bundle. For a spin 4-manifold with boundary $\tmfd$, the restriction of
either the positive or negative spinor bundle to $\tmfd$ is isomorphic to the
spinor bundle of $\tmfd$.
The analogue in dimension 4 of Corollary \ref{cor:ahat} is that
\begin{equation}
\label{eq:dim3}
e_\pm(X) = {\textstyle\frac{3}{4}}\sign(X) \pm \half \chi(X)
\end{equation}
for any closed spin 4-manifold $X$ (it suffices to check for $X = S^4$ and
$K3$). Recall Rokhlin's theorem that $\sign(X)$ is divisible by 16.

\begin{lem}
\label{lem:nu_su2}
Let $X$ be a closed 4-manifold with an $SU(2)$-structure
$(\omega^I, \omega^J, \omega^K)$ and $\tmfd$ a closed 3-manifold with a coframe
field $(e^1, e^2, e^3)$. Then
\[ \nu(\varphi) = 24\chi_2(\tmfd)\frac{\sign(X)}{16} \mmod 48 \]
for the product \gtstr{}
$\varphi = e^1 \wedge e^2 \wedge e^3 +
e^1 \wedge \omega^I + e^2 \wedge \omega^J+ e^3 \wedge \omega^K$
on $\tmfd \times X$.
\end{lem}

\begin{proof}
Pick a spin coboundary $W$ of $\tmfd$. Let $n_+(W,\pll)$ be the intersection
number with the zero section of a positive spinor field on $W$ whose
restriction to $\tmfd$ is the defining spinor field of the parallelism $\pll$
equivalent to the coframe field. We can apply connected sums with $T^4$ or
$S^2 \times S^2$ to make $n_+(W,\pll) = 0$ (this is the same argument as in
Lemma \ref{lem:spinbordism}), so we can assume that $\pll$ bounds an
$SU(2)$-structure on $W$.

If $X$ has an $SU(2)$-structure then $e_+(X) = 0$, so \eqref{eq:dim3} implies
$\chi(X) = -\frac{3}{2}\sign(X)$.
$W \times X$ is a $Spin(7)$-coboundary for~$\varphi$
so, applying Lemma \ref{lem:chimod2} in the final step,
\[ \nu(\varphi) = \chi(W \times X) - 3\sign(W \times X) 
= \big({-}24\chi(W) - 48\sign(W)\big)\frac{\sign(X)}{16} 
= 24\chi_2(\tmfd)\frac{\sign(X)}{16}  \mod 48 . \qedhere \]
\end{proof}

\subsection{Twisted connected sums} \label{subsec:tcs}

Now we sketch the basics of the twisted connected sum construction, ignoring
many details that are required to justify that the resulting \gtstr s
are torsion-free (see \cite{kovalev03, g2m}).
The construction starts from a pair of asymptotically cylindrical Calabi--Yau
3-folds~$V_\pm$. We can think of these as a pair of (usually simply connected)
6-manifolds with boundary $S^1 \times \kd_\pm$, for $\kd_\pm$ a K3 surface.
They are equipped with \sustr s $(\omega_\pm, \Omega_\pm)$ such that on a
collar neighbourhood $C_\pm \cong \hoi{0,1} \times \partial V_\pm$ of the
boundary 
\begin{equation}
\begin{aligned}
\omega_\pm &= dt \wedge d\anglen + \omega^I_\pm , \\
\Omega_\pm &= (d\anglen - idt) \wedge (\omega^J_\pm + i \omega^K_\pm) ,
\end{aligned}
\end{equation}
where $\anglen$ is the $S^1$-coordinate, $t$ is
the collar coordinate and $(\omega^I_\pm, \omega^J_\pm, \omega^K_\pm)$ is an
$SU(2)$-structure on~$\kd_\pm$.
The construction assumes that there is
a diffeomorphism $f : \kd_+ \to \kd_-$
such that
\[ f^*\omega^I_- = \omega^J_+, \; f^*\omega^J_- = \omega^I_+ \textrm{ and }
f^*\omega^K_- = -\omega^K_+ . \]
Now define \gtstr s on $S^1 \times V_\pm$ by
\[ \varphi_\pm = d\anglex \wedge \omega_\pm + \re \Omega_\pm, \]
where $\anglex$ denotes the $S^1$-coordinate, and a diffeomorphism
\begin{equation*}
\begin{aligned}
F : \; \partial(S^1 \times V_+) \cong
S^1 \times S^1 \times \kd_+ \; & \longrightarrow \;
S^1 \times S^1\times \kd_- \cong \partial(S^1 \times V_-) , \\
(\anglex, \anglen, x) \; & \longmapsto (\anglen, \anglex, f(x)) .
\end{aligned}
\end{equation*}
In the collar neighbourhoods $C_\pm$
\begin{equation*}
\varphi_\pm 
= d\anglex\wedge dt \wedge d\anglen + d\anglex\wedge\omega^I_\pm +
d\anglen\wedge\omega^J_\pm + dt \wedge \omega^K_\pm ,
\end{equation*}
so $\varphi_+$ and $\varphi_-$ patch up to a well-defined \gtstr{} $\varphi$
on the closed manifold 
\begin{equation} \label{eq:tcs}
M = (S^1 \times V_+) \cup_F (S^1 \times V_-). \end{equation}
One arranges that this \gtstr{} can be perturbed to a torsion-free one.
Because $F$ swaps the
circle factors at the boundary, $M$ is simply-connected if $V_+$ and $V_-$ are.

\subsection{A $Spin(7)$-bordism}

We now proceed with the proof of Theorem \ref{thm:nu24}, that the twisted
connected sum \gtstr s defined above always have $\nu = 24$.
Consider the diffeomorphism
\[ \tf = \Id \times {-\Id} \times f \; : \;
S^1 \times S^1 \times \kd_+ \to S^1 \times S^1 \times \kd_- , \]
and the ``untwisted connected sum''
$\tm = (S^1 \times V_+) \cup_{\tf} (S^1 \times V_-)$.
Then $\tm = S^1 \times N$, where $N = V_+ \cup_{-\Id \times f} V_-$.
Let $r$ denote the right angle rotation
$(\anglex, \anglen) \mapsto (\anglen, -\anglex)$ of $S^1 \times S^1$
and $g := F \circ \tf^{-1}$, and let $T_r$ and $T_g$ denote their mapping tori.
Then $g  = r \times \Id_\kd$, so $T_g \cong T_r \times \kd$.

To compute $\nu(\varphi)$ of the twisted connected sum \gtstr{} $\varphi$ on
$M$ and prove Theorem \ref{thm:nu24} we will construct a $Spin(7)$-bordism $W$
to product \gtstr s on $\tm \sqcup T_g$.
Let
\begin{align*}
B_\pm &= \left\{ (y-\half)^2 + t^2 < {\textstyle \frac{1}{16}} \right\}
\subset I \times S^1 \times C_\pm, \\ 
W_\pm &= I \times S^1 \times V_\pm \setminus B_\pm ,
\end{align*}
where $y$ denotes the $I$-coordinate, and $t$ the collar coordinate on
$C_\pm \subset V_\pm$ as before.
$\partial W_\pm$ is a union of five pieces, meeting in edges at
$\{y\}\times S^1 \times S^1 \times \kd$ for $y = 0, \quart, \threequart$ and 1:
a `top' and `bottom' piece each diffeomorphic to $S^1 \times V_\pm$,
$[0,\quart] \times S^1 \times S^1 \times \kd_\pm$ and
$[\threequart,1] \times S^1 \times S^1 \times \kd_\pm$, and
$E_\pm := \left\{ (y-\half)^2 + t^2 = \quart \right\}
\subset I \times S^1 \times C_\pm$.

We form a `keyhole' bordism $W$ by gluing some of these pieces: identify
$[0, \quart] \times S^1 \times S^1 \times \kd_\pm$ via $\Id \times \tf$,
and $[\threequart, 1] \times S^1 \times S^1 \times \kd_\pm$ via $\Id \times F$.
Then $\partial W$ is a disjoint union $M \sqcup \tm \sqcup T_g$, where
$M$~is formed by gluing the top pieces of $\partial W_+$ and $\partial W_-$ and
$\tm$ by gluing the bottom pieces, while the keyhole boundary component
$E_+ \cup E_-$ can be identified with the mapping torus $T_g$.

\newcommand{\insertpic}{
\begin{figure}

\psset{unit=1mm}
\begin{pspicture}(100,100)
\psset{origin={50,50}}
\newgray{intergray}{0.75}
\newgray{vlight}{0.85}
\newgray{darkgray}{0.5}
\pscustom[linestyle=none,fillstyle=solid,fillcolor=intergray]{
\psarc(0,0){15}{-90}{-45}
\psbezier(15,-15)(15,-30)(15,-40)
\psline(0,-40)}
\pscustom[linestyle=none,fillstyle=solid,fillcolor=vlight]{
\psarcn(0,0){15}{0}{-45}
\psbezier(15,-15)(15,-30)(15,-40)
\psline(30,-40)
\psbezier(30,-20)(25,0)(15,0)
}

\psset{linecolor=darkgray}
\psline(-50,-40)(-50,40)
\psline(50,-40)(50,40)
\psset{linecolor=black}
\psline(0,40)(0,15)
\psline(0,-40)(0,-15)
\psset{linewidth=2pt}
\psline(-50,-40)(50,-40)
\psline(50,40)(-50,40)
\pscircle(0,0){15}

\rput(50,50){%
\rput(-25,15){$W_+$}
\rput(25,15){$W_-$}
\rput(30,-3){$\rho = 0$}
\rput(7,-35){$\rho = \frac{\pi}{2}$}
\rput(0,45){$M$}
\rput(0,-45){$\tm = S^1 \times N$}
\rput(-7,3){\parbox{10mm}{\begin{gather*}T_g =\\T_r {\times} \kd\end{gather*}}}
\rput(-7,27){$\Id \times F$}
\rput(-7,-27){$\Id \times \tf$}
\rput(57, 40){$y = 1$}
\rput(57, -40){$y = 0$}
}
\end{pspicture}

\caption{The `keyhole' bordism $W$} 
\label{thepic}
\end{figure}
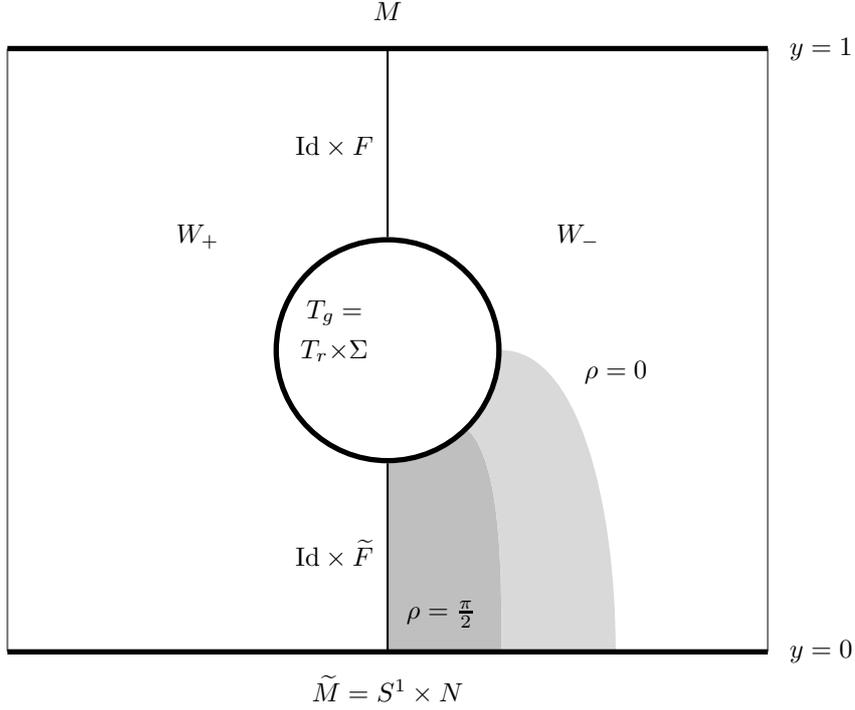}

\insertpic

It is easy to compute that $H_1(T_r) \cong \bbz \times \bbz_2$, so
$\chi_2(T_r) \equiv 1$. Since $\sign(\kd) = -16$, Lemma \ref{lem:nu_su2}
implies that any product \gtstr{} on $T_r \times \kd$ has $\nu = 24$,
while a product \gtstr{} on $\tm$ has $\nu = 0$.
To complete the calculation of $\nu(\varphi)$ it remains to show that
$W$ does indeed admit a suitable $Spin(7)$-structure, and to compute
the topological invariants of the $Spin(7)$-bordism~$W$.

\begin{lem}
$\chi(W) = 0$ and $\sign(W) = -16$.
\end{lem}

\begin{proof}
For the Euler characteristic, we use the usual inclusion-exclusion formula.
The spaces $W_+$, $W_-$ and $W_+ \cap W_-$ all contain $S^1$ factors, so
$\chi(W) = \chi(W_+) + \chi(W_-) - \chi(W_+ \cap W_-) = 0$.

For the signature, we must apply Wall's signature formula \cite{wall69}
because $W$ is formed by gluing $W_+$ and $W_-$ along only parts of boundary
components. 
The piece of the boundaries of $W_+$ and $W_-$ that we glue is
$X_0 = \left([0, \quart] \sqcup [\threequart, 1]\right) \times T^2 \times \kd$.
Let $Z = \partial X_0 =
\left\{0, \quart, \threequart, 1\right\} \times T^2 \times \kd$ (the edges of
$\partial W_\pm$), and
\[ X_\pm := \partial(W_\pm) \setminus X_0 = \bigl( \{0,1\} \times S^1 \times V_\pm \bigr)
\; \sqcup \; E_\pm , \]
where $E_\pm$ are the keyhole pieces as defined above.

Throughout this proof we will use real coefficients for all cohomology groups.
We need to identify the images $A$, $B$ and $C$ in $H^3(Z)$ of $H^3(X_0)$,
$H^3(X_+)$ and $H^3(X_-)$, respectively; each is a Lagrangian subspace with
respect to the intersection form $(\,\,,\,)$ on $H^3(Z)$.
The vector space $K = \frac{A \cap (B+C)}{(A \cap B) + (A \cap C)}$ admits the
following natural non-degenerate symmetric bilinear form $q$:
if $a, a' \in A \cap (B+C)$ (representing $[a], [a'] \in K$)
and $a' = b' + c'$, $b' \in B$, $c' \in C$, then we set
\[ q([a], [a']) := -(a,b'). \]
Since $W_\pm$ both have signature 0, the
signature formula \cite[Theorem p.\,271]{wall69}
implies that the signature of $W$ equals the signature of $(K,q)$.

We can identify $Z_y := \{y\} \times T^2 \times \kd$ with
$S^1 \times \partial V_+$. On $Z_y$, let $\anglex$ denote the coordinate on the
$S^1$ factor from $S^1 \times V_+$, and $\anglen$ the coordinate on the $S^1$
factor in $\partial V_+$. Let $\theta_+ = [d\anglex]$ and
$\theta_- = [d\anglen] \in H^1(Z_y)$. If $w \in H^4(\kd)$ is positive then
$\theta_+ \wedge \theta_- \wedge w \in H^6(Z_y)$ is positive with respect to
the orientation on $Z_y$ given by the identification with
$S^1 \times \partial V_+$.
The orientation on $Z$ that we should use to define its intersection form in
the application of the signature formula is that induced as the boundary
of $X_+$, \ie
\[ Z = Z_1 \sqcup - Z_\threequart \sqcup \; Z_\quart \sqcup - Z_0 . \]
Since the K3 surface $\kd$ has no cohomology in odd degrees, the vector space
$H^3(Z)$ decomposes as the sum of 8 copies of $L := H^2(\kd)$: we let $L_{y\pm}$
denote the image of $L \to H^3(Z_y), \; \ell \mapsto \theta_\pm \wedge \ell$.
(This means for example that if $\alpha_\pm \in H^2(V_\pm)$ then the
restriction of $[d\anglex] \wedge \alpha_\pm \in H^3(W_\pm)$ to $Z_y$ lies in
$L_{y+}$ for $y = 0, \quart$, and in $L_{y\pm}$ for $y = \threequart, 1$.)
For $\hz \in H^3(Z)$, let $\hz_{y\pm} \in L$ denote the $L_{y\pm}$ component
under this isomorphism.
Then the intersection form on $H^3(Z)$ is given in terms of the inner
product $\inner{\,\,,\,}$ on $L$ by
\begin{equation}
\label{eq:intZ}
\begin{aligned}
(\hz,\hz') &= \inner{\hz_{1+}, \hz'_{1-}} - \inner{\hz_{1-}, \hz'_{1+}}
- \inner{\hz_{\threequart+}, \hz'_{\threequart-}}
+ \inner{\hz_{\threequart-}, \hz'_{\threequart+}} \\
&+ \inner{\hz_{\quart+}, \hz'_{\quart-}}
- \inner{\hz_{\quart-}, \hz'_{\quart+}}
- \inner{\hz_{0+}, \hz'_{0-}} + \inner{\hz_{0-}, \hz'_{0+}} .
\end{aligned}
\end{equation}
Let $N_\pm$ denote the image of $H^2(V_\pm)$ in $H^2(\kd) \cong L$,
and $T_\pm \subset L$ the orthogonal complement. By the long exact sequence
of the pair $(V_+, S^1 {\times} \kd_+)$ and Poincar\'{e}-Lefschetz duality, the
image of $H^3(V_+)$ in $H^3(S^1 \times \kd)$ is the annihilator of the image of
$H^2(V_+)$ under the intersection pairing,
which equals $[d\anglen] \wedge T_+$. We find that
\begin{equation}
\label{eq:abc}
\begin{aligned}
A &= \{ \hz \in H^3(Z) :
\hz_{0\pm} = \hz_{\quart \pm}, \;
\hz_{\threequart \pm} = \hz_{1 \pm} \}, \\
B &= \{ \hz \in H^3(Z) : \hz_{0+}, \hz_{1+} \in N_+, \;
\hz_{0-}, \hz_{1-} \in T_+, \;
\hz_{\quart \pm} = \hz_{\threequart \pm} \}, \\
C &= \{ \hz \in H^3(Z) : \hz_{0+}, \hz_{1-} \in N_-, \;
\hz_{0-}, \hz_{1+} \in T_-, \;
\hz_{\quart \pm} = \pm \hz_{\threequart \mp} \}.
\end{aligned}
\end{equation}

Given an element of $K$ represented by $a = b + c$, we can certainly find
some $h \in A \cap B$ with $h_{1\pm} = b_{1\pm}$. Replacing $a$ by $a-h$, we
may assume without loss of generality that $b_{1\pm} = 0$. Similarly we
can assume $c_{1\pm} = 0$, and then $a_{1\pm} = 0$ too.
Setting
\[ n := a_{0+}, \; t := a_{0-}, \; n_+ := b_{0+}, \; t_+ := b_{0-},
\; n_- := c_{0+}, \; \textrm{ and }t_- := c_{0-}, \]
the remaining components are determined by \eqref{eq:abc}. Thus we find that
any element of $K$ can be represented by $a = b+c$ such that
\[ a = \begin{pmatrix} 0 & 0 \\ 0 & 0 \\  n & t \\ n & t \end{pmatrix},
\quad
b = \begin{pmatrix} 0 & 0 \\ \frac{n+t}{2} & \frac{-n+t}{2} \\
\frac{n+t}{2} & \frac{-n+t}{2} \\ n_+ & t_+ \end{pmatrix},
\quad
c = \begin{pmatrix} 0 & 0 \\ \frac{-n-t}{2} & \frac{n-t}{2} \\
\frac{n-t}{2} & \frac{n+t}{2} \\ n_- & t_- \end{pmatrix}  \]
(where the top left matrix entry corresponds to the $1+$ component etc), and
\[ n_\pm \in N_\pm, \quad t_\pm \in T_\pm, \quad
n = n_+ + n_-, \quad t = t_+ + t_- . \]
Representing a pair of classes $[a], [a'] \in K$ by elements of that form,
applying \eqref{eq:intZ} and rearranging gives
\begin{equation}
\label{eq:wall}
\begin{aligned}
2q([a], [a']) = -2(a,b') &= -\inner{n, -n'{+}t'} + \inner{t, n'{+}t'}
+ \inner{n, 2t'_+} - \inner{t, 2n'_+} \\
&= \inner{n, n'} + \inner{t, t'} +
\inner{n, t'_+ {-} t'_-} + \inner{t, -n'_+ {+} n'_-} .
\end{aligned}
\end{equation}
Now consider
\begin{align*}
K_0   &= \{ [a] \in K : n \in N_+ \cap N_-, \; t \in T_+ {+} T_- \}, \\
K_\pm &= \{ [a] \in K : n = t \in N_\pm \cap (T_+ {+} T_-) \}.
\end{align*}
If we use \eqref{eq:wall} to evaluate the product of two elements of $K_0$
then the cross terms $\gen{n, t'}$ etc vanish, and
$q([a], [a']) = \gen{n,n'} + \gen{t, t'} = \gen{n+t, n'+t'}$. Hence $K_0$ is
isometric to $L$, so has signature $-16$.

If $[a] \in K_+$, then the RHS of \eqref{eq:wall} reduces to $2\gen{t, n'_-}$,
which vanishes if $[a'] \in K_0 + K_+$. Similarly $K_-$ is orthogonal to
$K_0 + K_-$. This implies in particular that $K_+$ and $K_-$ are transverse,
and since $K_+ \oplus K_-$ is a sum of isotropic spaces it has signature 0.

Finally, note that $K_+ \oplus K_-$ is a complement to $K_0$ in $K$: given
$(n,t) \in (N_+ + N_-) \times (T_+ + T_-)$ we can certainly subtract
an element of $N_+ \cap N_-$ from $n$ to ensure that $n \in T_+ + T_-$, and
then an element of $T_+ + T_-$ from $t$ to ensure $n = t$. Hence the orthogonal
complement to $K_0$ is precisely $K_+ \oplus K_-$, and
\[ \sign(W) = \sign(K) = \sign(K_0) + \sign(K_+ \oplus K_-) = -16 . \qedhere \]
\end{proof}

To finish the proof of Theorem \ref{thm:nu24}, we need to exhibit a
$Spin(7)$-structure on $W$ with the right restrictions to the boundary
components: the restriction to $M$ should be the twisted connected sum
\gtstr{}~$\varphi$, while the restrictions to $\tm = S^1 \times N$ and
$T_g = T_r \times \kd$ should be product \gtstr s.
We can define an $SU(3)$-structure on $N$ as follows. Let $V'_-$ be the
complement of the collar neighbourhood $C_- \subset V_-$. On $C_-$ set
\begin{align*}
\omega' &= dt \wedge d\anglen + c_\rho \omega^I_- + s_\rho \omega^J_- , \\
\Omega' &= (d\anglen - idt) \wedge
(c_\rho \omega^J_- - s_\rho \omega^I_- + i \omega^K_-) ,
\end{align*}
where $c_\rho = \cos \rho, \; s_\rho = \sin \rho$ for a smooth function $\rho$
supported on $C_-$, such that $\rho = \frac{\pi}{2}$ on $\partial V_-$.
Take $\tomega$ to be $\omega_+$ on $V_+$, $\omega'$ on $C_-$, and $\omega_-$
on $V'_-$, and define $\tOmega$ analogously. Then $(\tomega, \tOmega)$ is a
well-defined \sustr{} on $N$, and $\tv = d\theta \wedge \tomega + \re \tOmega$
is a product \gtstr{} on $\tm$. 

Next we define the \spstr{} $\psi$ on $W$.
Let $y$ be the $I$ coordinate on each half.
First, define $\rho$ on $I \times C_-$ to be $\frac{\pi}{2}$ on a
neighbourhood of $[0, \quart] \times \partial V_-$ and have compact
support in $\hoi{0,\half} \times C_-$ (see Figure \ref{thepic}),
and use this to define forms $\tomega$ and $\tOmega$ on $I \times V_-$.
Since $dy$ is a global covector field on~$W_0$, defining a \spstr{} is
equivalent to defining a \gtstr{} on each slice $y =$ const. Take this
to be $\varphi_+ = d\theta \wedge \omega_+ + \re \Omega_+$ on
$\{y\} \times S^1 \times V_+$, and $d\theta \wedge \tomega + \re \tOmega$
on $\{y\} \times S^1 \times V_-$.
Then the restriction of $\psi$ to the boundary components $M$ and $\tm$ are
$\varphi$ and $-\tv$ respectively, as desired.

Finally we show that the restriction of $\psi$ to the `keyhole' boundary component
$T_g = E_+ \cup E_-$ is a product \gtstr{} too.
We first outline the argument, starting from
$E_\pm \cong [0, \pm \pi] \times S^1 \times S^1 \times \kd_\pm$ (the first factor
corresponding to one half of the circle $\{(y-\half)^2 + t^2 = \frac{1}{16}\}$)
being embedded as a product inside $I \times C_\pm$.
The restriction of $\psi$ to $I \times C_\pm$ is a product of two
$SU(2)$-structures, so the induced \gtstr{} on $E_\pm$ is a product of a
coframe field on $[0, \pm \pi] \times S^1 \times S^1$ and an $SU(2)$-structure
on $\kd$.
The coframes on the two copies of $[0,\pm \pi] \times S^1 \times S^1$ patch up
to a coframe on their union $T_r$, and the \gtstr{} on $T_g$ is the product of
that with an $SU(2)$-structure on~$\kd$.

In order to fill in the details of this sketch we need to write down the
structures explicitly, which is rather cumbersome.
To make the notation slightly more manageable we will use a complex form
as a shorthand for an ordered pair of real forms, so that an $SU(2)$-structure
can be defined by one complex and one real 2-form, or a coframe field on
a 3-manifold by one complex and one real 1-form.
Also, we identify both $\kd_+$ and $\kd_-$ with
a standard K3 surface $\kd$, so that $f$ corresponds to $\Id_\kd$.
Setting $y = -\half c_\alpha + \half, \; t = \half s_\alpha$
for $\alpha \in [0,\pi]$ lets us identify $E_+ \subset I \times C_+$ with
$[0,\pi] \times S^1 \times S^1 \times \kd$.
On $I \times C_+$, $\psi$ is the product of the $SU(2)$-structure
\begin{equation}
\label{eq:su2tr}
\big( (dy - idt) \wedge (d\anglex + i d\anglen), \;
dy \wedge dt - d\anglex \wedge d\anglen \big)
\end{equation}
on $I \times \hoi{0,1} \times S^1 \times S^1$ and
$(\omega^I_+ +i\omega^J_+, \; \omega^K_+)$ on $\kd$.
The induced \gtstr{} on $E_+$
is given by contraction with the normal vector field
$c_\alpha \contra{y} - s_\alpha \contra{t}$. The result is the product of
the same $SU(2)$-structure on $\kd$ with the coframe field
$(e^{i\alpha}(d\anglex +i d\anglen), \; \half d\alpha)$ on
$[0,\pi] \times S^1 \times S^1$.

Similarly, for $\alpha \in [\pi, 2\pi]$
we set $y = -\half c_\alpha + \half, \; t = -\half s_\alpha$ to identify
$[\pi,2\pi] \times S^1 \times S^1 \times \kd_- \cong E_-$.
On $I \times C_-$, the restriction of $\psi$ is given by the product
of \eqref{eq:su2tr} on $I \times \hoi{0,1} \times S^1 \times S^1$
and $\left(e^{-i\rho} (\omega^I_- + i \omega^J_-), \; \omega^K_-\right)$ on the
tangent space to the $\kd$ factor. Contracting with the normal vector field
$c_\alpha \contra{y} + s_\alpha \contra{t}$ gives the coframe
$\left(e^{-i\alpha}(d\anglex + id\anglen), \; -\half d\alpha \right)$ on
$[\pi,2\pi] \times S^1 \times S^1$. 
Now, as product \gtstr s
\begin{multline*}
\left(e^{-i\alpha}(d\anglex + id\anglen), \; -\half d\alpha \right) \cdot 
\left(e^{-i\rho} (\omega^I_- + i \omega^J_-), \; \omega^K_-\right) = \\
\left(e^{i(\rho-\alpha)}(d\anglex + id\anglen), \; -\half d\alpha \right)
\cdot (\omega^I_- + i \omega^J_-, \; \omega^K_-) = 
\left(e^{i(\alpha-\rho)}(d\anglen + id\anglex), \; \half d\alpha \right)
\cdot (\omega^I_+ + i \omega^J_+, \; \omega^K_+) .
\end{multline*}

$T_g$ is formed by gluing the boundaries of
$[0,\pi] \times S^1 \times S^1 \times \kd$ and
$[\pi,2\pi] \times S^1 \times S^1 \times \kd$ using
$(\pi, \anglex, \anglen, x) \mapsto (\pi, \anglen, \anglex, x)$
and $(0, \anglex, \anglen, x) \mapsto (2\pi, \anglex, -\anglen, x)$.
These maps preserve the $SU(2)$-structure on the $\kd$ factor, and match up
the coframes $(e^{i\alpha}(d\anglex +i d\anglen), \; \half d\alpha)$ and
$\left(e^{i(\alpha-\rho)}(d\anglen + id\anglex), \; \half d\alpha \right)$ to a
well-defined coframe on $T_r$ (since $\rho = 0$ at $\alpha = \pi$ and
$\rho = \frac{\pi}{2}$ at $\alpha = 0$, $2\pi$).
Thus the \gtstr{} on $T_g = T_r \times \kd$ is a product, completing the proof
of Theorem \ref{thm:nu24}.

\subsection{Orbifold resolutions} \label{subsec:orbifolds}

For some of Joyce's examples of compact \gtmfd s constructed by resolving flat
orbifolds, the torsion-free \gtstr s are homotopic to twisted connected sum
\gtstr s, and thus have $\nu = 24$. It is proved in \cite{jn2}
that in some cases there is even a connecting path of torsion-free \gtstr s,
but that is irrelevant for the calculation of $\nu$.

We have no general technique for computing $\nu$ of orbifold resolution
\gtmfd s. We note, however, that a small number of examples have
$b_2(M) + b_3(M)$ even, \eg \cite[\S 12.8.4]{joyce00}. Those \gtmfd s 
have $\chi_\Q(M)$---and hence $\nu$---odd.

\section{The $h$-principle for coclosed \gtstrsoft s}
\label{sec:hp}

We now prove Theorem \ref{thm:hp}, that coclosed \gtstr s satisfy the
h-principle. We first set up some notation, continuing from
\S \ref{subsec:strs}.

\subsection{Positive 4-forms}

For a vector space $V$ of dimension $7$, let $\Lambda^3_+V^*$ and
$\Lambda^4_+V^*$ denote the space of forms equivalent to $\varphi_0$ (as
defined in \eqref{g2formeq}) and $*\varphi_0$ respectively. These are
\emph{open} subsets of the spaces of forms. Any
$\varphi \in \Lambda^3_+V^*$ defines a \gtstr, and thus an inner product and
orientation, and a Hodge star operator. This gives a non-linear map
$\Lambda^3_+V^* \to \Lambda^4_+V^*, \varphi \mapsto *\varphi$, which is 2-to-1.
The stabiliser of a $\sigma \in \Lambda^4_+V^*$ is isomorphic to
$G_2 \times \{\pm 1\}$,
so $\sigma$ \emph{together with} a choice of orientation
on $V$ determines a \gtstr{} \cite[\S 2.8.3]{bryant03}.

We say that a \gtstr{} on a 7-manifold $M$, defined by a positive 3-form
$\varphi \in \Sec \Lambda^3_+(M)$, is coclosed if the associated 4-form
$\sigma = *\varphi \in \Sec \Lambda^4_+(M)$ is closed. The set of coclosed
\gtstr s on an oriented manifold $M$ is therefore the same as the space of
closed positive 4-forms $\Clo \Lambda^4_+(M) \subset \Sec \Lambda^4_+(M)$.
(Each section induces a spin structure, and the space $\GTC(M)$ appearing in
the statement of Theorem \ref{thm:hp} is a subset of $\Clo \Lambda^4_+(M)$
compatible with a fixed spin structure on $M$.)

\subsection{Microextension}

It is generally easier to prove h-principles for partial differential relations
on open manifolds than on closed manifolds. The Hirsch microextension trick is
the strategy to prove h-principles on closed manifolds by reducing the problem
to an h-principle on an open manifold of higher dimension.

In order to apply the microextension trick, we consider 4-forms on 8-manifolds
such that the restriction to every hypersurface is a positive 4-form. The key
point that makes the argument work is that not only is the set of such forms
open, but moreover any positive 4-form from a hyper\-surface can be extended
this way. This is the feature that enables us to prove the h-principle for
coclosed \gtstr s on closed manifolds, but not for, say, symplectic structures
or closed \gtstr s.

\begin{defn}
For a vector space $W$ of dimension 8, let
\[ \calr(W) = \{\chi \in \Lambda^4 W^* \,|\, \chi_{|V} \in \Lambda^4_+V^*
\textrm{ for every hyperplane } V \subset W \} . \]
\end{defn}

If $W = V \oplus \bbr$ and $\varphi \in \Lambda^3_+V^*$ then
the invariance of $\psi = dt \wedge \varphi + *\varphi$ under $Spin(7)$
(\cf \eqref{eq:product_sp}), which acts transitively on the hyperplanes,
shows that $\psi \in \calr(W)$.

\begin{lem}
$\calr(W)$ is open in $\Lambda^4 W^*$.
\label{lem:r_open}
\end{lem}

\begin{proof}
Let $G \cong \bbr P^7$ denote the Grassmannian of hyperplanes in $W$, and
$\pi : \calv \to G$ the tautological bundle. If
$f : \pi^{-1}(U) \to U \times \bbr^7$ is a local trivialisation,
then $\Lambda^4 W^* \times U \to \Lambda^4(\bbr^7)^*$,
$(\chi, V) \mapsto f_{V*}(\chi_{|V})$ is continuous, so the pre-image of
$\Lambda^4_+(\bbr^7)^*$ is open. Hence if $\chi \in \calr(W)$ then for each
$V \in G$ there are open neighbourhoods $B_V \subset \Lambda^4W^*$ of $\chi$
and $C_V \subset G$ of $V$ such that $\chi'_{|V'} \in \Lambda^4_+{V'}^*$ for
each $\chi' \in B_V$ and $V' \in C_V$. Since $G$ is compact it can be covered
by $C_{V_1}, \ldots, C_{V_k}$ for finitely many $V_1, \ldots, V_k \in G$. Then
$B_{V_1} \cap \cdots \cap B_{V_k}$ is an open neighbourhood of $\chi$ in
$\Lambda^4W^*$ and contained in~$\calr(W)$.
\end{proof}

For an 8-manifold $N$, let $\calr(N) \subset \Lambda^4(N)$ be the subbundle
with fibres $\calr(T_xN) \subset \Lambda^4T^*_xN$.
Let $\Clo \calr(N) \subset \Sec \calr(N)$ denote the subspace of closed
4-forms, and $\Clo_a \calr(N)$ the subspace of forms representing a fixed
cohomology class $a \in H^4_{dR}(N)$.
Because the subbundle $\calr(N) \subset \Lambda^4(N)$ is open and invariant
under the natural action of $\Diff(N)$, 
\cite[Theorem 10.2.1]{eliashberg02} immediately implies that
$\Clo_a \calr(N) \into \Sec \calr(N)$ is a homotopy equivalence
if $N$ is an open manifold.

\subsection{The proof of Theorem \ref{thm:hp}}

We prove the following stronger version of Theorem \ref{thm:hp}.

\begin{thm}
Let $I^k \to \Sec \Lambda^4_+(M), \; s \mapsto \sigma_s$ and
$I^k \to H^4_{dR}(M), \; s \mapsto a_s$ be families such that
$\sigma_s \in \Clo_{a_s} \Lambda^4_+(M)$ 
for all $s \in \partial I^k$. Then the family $\sigma_s$ is
homotopic in $\Sec \Lambda^4_+(M)$, relative to~$\partial I^k$, to a family
$\sigma'_s$ such that $\sigma'_s \in \Clo_{a_s} \Lambda^4_+(M)$
for all $s \in I^k$.

In particular
\begin{itemize}[labelindent=\parindent]
\item $\Clo \Lambda^4_+(M) \into \Sec \Lambda^4_+(M)$ is a homotopy
equivalence;
\item $\Clo_a \Lambda^4_+(M) \into \Sec \Lambda^4_+(M)$ is a homotopy
equivalence for each fixed $a \in H^4_{dR}(M)$.
\end{itemize} 
\end{thm}

\begin{proof}
Identify $\sigma_s$ with its pull-back to $M \times \bbr$, and
let $\chi_s = \sigma_s + dt \wedge *\sigma_s - td(*\sigma_s) \in
\Sec \Lambda^4(M \times \bbr)$. Then there is $\epsilon > 0$ such that $\chi_s$
takes values in $\calr$ over $N := M \times (-\epsilon, \epsilon)$ for all
$s \in I^k$, and $\chi_s \in \Clo_{a_s} \calr(N)$ for $s \in \partial I^k$.
If $a_s \equiv a$ is constant in $s$ then it follows immediately from
\cite[Theorem 10.2.1]{eliashberg02} that the family $\chi_s$ is homotopic
in $\Sec \calr(N)$, relative to~$\partial I^k$, to a family
$\chi'_s \in \Clo_a \calr(N)$. If we set $\sigma'_s = \chi'_{s|M}$ then
$\sigma'_s \in \Clo_a \Lambda^4_+(M)$ for all $s \in I^k$, and the
restriction to $M$ of the homotopy from $\chi$ to $\chi'$ gives a homotopy
from $\sigma$ to $\sigma'$ in $\Sec \Lambda^4_+(M)$.

The proof of \cite[Theorem 10.2.1]{eliashberg02} builds on
\cite[Proposition 4.7.4]{eliashberg02}, which is stated for the case when
$a_s$ is constant. However, the proof still works if $a_s$ is allowed to depend
on $s$ (\cf \cite[Exercise in \S 10.2]{eliashberg02}).
\end{proof}

\section{The action of spin diffeomorphisms on $\GTH(M)$}
 \label{sec:diffeomorphisms}
Let $(M, \varphi) $ be a closed connected spin $7$-manifold with
\gtstr{}.  In this section we investigate the action
 of the group of spin diffeomorphisms of $M$ on the set of homotopy 
 classes of \gtstr s on $M$:
 \[ \GTH(M) \times \DiffSpin(M) \to \GTH(M), \quad ([\varphi], f) \mapsto [f^*\varphi] .\]
The quotient is the set $\GTD(M)$ of deformation classes of \gtstr s.
To determine the action for a specific spin diffeomorphism $f \colon M \cong M$
amounts to computing the difference class
 $ D(\varphi, f^*\varphi) \in \Z. $
The existence of the $\nu$-invariant ensures that $D(\varphi, f^*\varphi) = 24 k$ 
for some integer $k$.  In this section we relate the possible values of $k$
to the topology of $M$ and in particular $p_M \in H^4(M)$. At the end we
provide  the general definition of the $\xi$-invariant.

\subsection{The spin characteristic class $p_M$}
\label{subsec:p_1/2}

Recall that the classifying space $BSpin$ is $3$-connected and
$\pi_4(BSpin) \cong \Z$.
It follows that $H^4(BSpin) \cong \Z$ is infinite cyclic.  A generator is
denoted $\pm \frac{p_1}{2}$ and the notation is justified since for the
canonical map $\pi \colon BSpin \to BSO$ we have $\pi^*p_1 = 2 \frac{p_1}{2}$
where $p_1$ is the first Pontrjagin class.
Given a spin manifold $X$ we write
\[ p_{X} : = \frac{p_1}{2}(X) \in H^4(X) .\]
The following lemma is well known to experts.

\begin{lem}[{\cite[Lemma 2.2(i)]{7class}}]
\label{lem:p_1/2}
For a closed spin $7$-manifold $M$, $p_M \in 2 H^4(M)$.
\end{lem}

For later use, we recall from the introduction that $d_\pi$ denotes the greatest divisor of $p_M$
modulo torsion, while
$\mdiv :=
{\rm Max} \{ s \,|\,s, \mpl \in \Z, \; \mpl^2 s \text{~divides~} \mpl p_M\}$;
we set $d_\pi = \mdiv = 0$ if $p_M$ is torsion. Both are even by
Lemma \ref{lem:p_1/2}.

\begin{ex}
If $H^4(M) \cong \Z \oplus \Z_4$ and $p_M \mapsto (8,2)$ then $d_\pi = 8$ while
$\mdiv = 4$.
\end{ex}
%


\subsection{Translations of \gtstrsoft s and mapping tori}
\label{subsec:mapping_tori}
Given $(M, \varphi)$ and a spin diffeo\-morphism $f \colon M \cong M$, we wish
to calculate the difference element $D(\varphi, f^*\varphi) \in \Z$. 
Note that (given $\varphi$) the homotopy class $[f^*\varphi] \in \GTH(M)$
depends only on the pseudo-isotopy class of $f$. For suppose that $F$ is
pseudo-isotopy between diffeomorphisms $f_0$ and $f_1$, \ie 
a diffeomorphism $F \colon M \times I \cong M \times I$ such that
$F|_{M \times \{ i \} } = f_i$ for $i = 0, 1$. Then contracting 
the pull-back $F^*\psi$ of the product $Spin(7)$-structure
$\psi = dt \wedge \varphi + * \varphi$ with $\contra{t}$ and restricting
to the slices $M \times \{t\}$ defines a homotopy between $f_0^*\varphi$ and
$f_1^*\varphi$.
On the other hand, Proposition \ref{prop:Dandp} shows that
$D(\varphi, f^*\varphi)$ does not depend upon the \gtstr{}~$\varphi$.
Hence we obtain a well-defined function
\[ D_M \colon \wt \pi_0 \DiffSpin(M) \to \Z, \quad [f] \mapsto D_M(f) : = D(\varphi,
f^*\varphi),\] 
where $\wt \pi_0 \DiffSpin(M)$ denotes the group of
pseudo-isotopy classes of spin diffeomorphisms of $M$.

The integer $D_M(f)$ measures the translation action of $f$ on the set of
homotopy classes of \gtstr s.
Next we show how to calculate $D_M(f)$ using the mapping torus of $f$:
\[ T_f : = (M \times [0, 1])/ (x,0) \sim (f(x),1) . \]
\noindent
Since $f$ is a spin diffeomorphism the closed $8$-manifold $T_f$ admits a spin structure.  We choose
a spin structure and let $T_f$ to denote the corresponding $8$-dimensional spin manifold: no confusion shall arise since we are interested only in the characteristic number
\[  p^2(f) : = \langle p_{T_f}^2, [T_f] \rangle \in \Z, \]
which depends only on the oriented diffeomorphism type of $T_f$  since
$2p_{T_f} = p_1(T_f)$ and $H^8(T_f) \cong \Z$ (in fact $p_{T_f}$ is independent
of the choice of spin structure by \cite[p.\,170]{cadek08}).
Therefore $p^2(f)$
is an invariant of the pseudo-isotopy class of $f$ and we define the function
\[ p^2 \colon \wt \pi_0 \DiffSpin(M) \to \Z, \quad [f] \mapsto p^2(f). \]
%
%
The following proposition proves Proposition \ref{prop:Dshift} and shows how the
mapping torus $T_f$ can be used to compute the difference class $D(\varphi, f^*\varphi)$.

\begin{prop}
The function $D_M \colon  \wt \pi_0 \DiffSpin(M) \to \Z $ is a homomorphism
given by 
\label{prop:Dandp}
\[  D(\varphi, f^*\varphi)  \; = \; \frac{3 \cdot p^2(f)}{28}  \; = \;
24 \ahat(T_f).  \]
\end{prop}

\begin{proof}
From the definition of $D(\varphi, \varphi')$ in \S \ref{sec:invariant} it is clear that
$D(f^*\varphi, f^*\varphi') = D(\varphi, \varphi')$ for any spin diffeomorphism $f$
and any pair of \gtstr s $\varphi$ and $\varphi'$ on $M$.  Now for two spin
diffeomorphisms $f_0, f_1 \colon M \cong M$, 
the affine property \eqref{eq:d_affine} of $D$ gives
\[ D(\varphi, (f_1 \circ f_0)^* \varphi)
\; = \; D(\varphi, f_0^*\varphi) + D(f_0^*\varphi, f_0^*(f_1^*\varphi)) 
\; = \; D(\varphi, f_0^*\varphi) + D(\varphi, f_1^*\varphi). \]
This shows that $D_M$ is a homomorphism.

Turning to the mapping torus, we can use Lemma \ref{lem:d_speuler}
to compute $D(\varphi, f^*\varphi)$ by treating the product $M \times [0,1]$ together with
the embeddings $(\Id, 0)$ and $(f, 1) : M \into M \times [0,1]$ as a
$Spin(7)$-bordism $W_f$ from $(M, f^*\varphi)$ to $(M, \varphi)$. 
Clearly the manifold $\ow_f$ obtained by closing up the bordism as in
\eqref{eq:close} is nothing other than the mapping torus $T_f$,
so \eqref{eq:d_speuler} gives 
\[ D(f^*\varphi, \varphi)  \; = \; - e_+(\ow_f)  \; = \; -e_+(T_f) .\]
By Proposition \ref{prop:e+}, $e_+(T_f) = \frac{1}{16}(4p_{T_f}^2 - 4p_2 + 8e)$ and using the signature theorem to eliminate $p_2$ from this equation we have
\[ D(\varphi, f^*\varphi)  \; = \; e_+(T_f)  \; = \;
\frac{3p_{T_f}^2}{28} - \frac{45 \sign(T_f)}{28} + \frac{\chi(T_f)}{2} .\]
Since $T_f$ is a mapping torus both $\sign(T_f)$ and $\chi(T_f)$ vanish
which proves the first equality of the proposition.
Now the second equality follows from Corollary \ref{cor:ahat}.
\end{proof}

Since Proposition \ref{prop:Dandp} determines $D_M$ in terms of $p^2$,
the proofs of Theorems \ref{thm:spinauto_LB} and \ref{thm:spinauto_UB}
are completed by quoting the following result. Here $b_M$ denotes the torsion
linking form on $\Tor H^4(M)$.

\begin{thm}[{\cite[Definition 4.4 and Corollary 4.17(iv)]{7class}}]
\label{thm:rdiff}
For any spin 7-manifold $M$, there is an $r \in \{0,1,2\}$ depending only
on $(H^4(M), b_M, p_M)$ such that
\begin{equation}
\label{eq:rdiff}
p^2(\DiffSpin(M)) \subseteq \lcm(224, 2^r\mdiv(M))\Z ,
\end{equation}
with equality if $M$ is 2-connected.
\end{thm}

The next subsection summarises some ingredients of the proof of this theorem.
However, before we do so let us prove an elementary special case of
\eqref{eq:rdiff} in order
to make the appearance of $\dinfty{M}$ less mysterious.

\begin{lem} \label{lem:D_constraints}
Let $M$ be a closed spin $7$-manifold and $f$ a spin diffeomorphism of $M$.
Then
\begin{equation}
p^2(f) \in \lcm(224, \dinfty{M})  \Z. 
\end{equation}
%
\end{lem}

\begin{proof}
First recall \eqref{eq:ek_ahat}: for a closed $8$-dimensional spin manifold $X$,
combining the definitions
\eqref{eq:genera} of the $L$-genus and the $\ahat$-genus gives
\[ p_X^2 - \sigma(X) = 8 \cdot 28 \ahat(X) . \]
Since the mapping torus $T_f$ is a closed \mbox{$8$-dimensional} spin manifold
with $\sigma(T_f) = 0$ we deduce that
\begin{equation} \label{eq:div2}
p_{T_f}^2 \in 8 \cdot 28 \cdot \Z.
\end{equation}
From the definition of $\mdiv$ there is a positive
integer $\mpl$ such that $\mpl^2 \mdiv$ divides $\mpl p_M$.  Applying
Lemma \ref{lem:d-p^2} below with $x = \mpl p_{T_f}$ and $s = \mpl^2 \dinfty{M}$
gives that $\mpl^2\dinfty{M}$ divides $\mpl^2p_{T_f}^2$ and hence
\begin{equation*} 
p^2_{T_f} \in \dinfty{M} \cdot  \Z. \qedhere
\end{equation*}
\end{proof}

\begin{lem} \label{lem:d-p^2}
Let $T_f$ be the mapping torus of $f \colon M \cong M$ and $i \colon M \to T_f$
the inclusion.
If $x \in H^4(T_f)$ and $s \in \Z$ divides $i^*x$ then $s$
divides $x^2 \in H^8(T_f) \cong \Z$.
%
\end{lem}

\begin{proof}
Consider the following fragment of the long exact cohomology sequence for the mapping torus $T_f$ with $\Z_s$ coefficients:
\begin{equation*}
H^3(M; \Z_s) \xrightarrow{\Id - f^*} H^3(M; \Z_s) \xrightarrow{~\del~} H^4(T_f; \Z_s) \xrightarrow{~i^*~} 
H^4(M; \Z_s) \xrightarrow{\Id - f^*} H^4(M; \Z_s). 
\end{equation*}
For a space $X$, let $\rho_s \colon H^*(X) \to H^*(X; \Z_s)$ denote reduction mod~$s$.  By assumption $i^*\rho_s(x) = 0 $ and so $\rho_s(x)$ lies in the image of $\del$.  But the cup-product 
\[ H^4(T_f; \Z_s) \times H^4(T_f; \Z_s) \to \Z_s \]
vanishes on $\im(\del)$.  Hence $\rho_s(x)^2 = \rho_s(x^2) = 0 \in H^8(T_f; \Z_s)$ and so $s$ divides $x^2$.
%
%
\end{proof}

\subsection{Diffeomorphisms of spin 7-manifolds}

We now summarise the main ideas of the proof of Theorem \ref{thm:rdiff}
from \cite{7class}. We also recall from there the technical notion of a Gauss
refinement and how that detects some aspect of the action of diffeomorphisms
and almost-diffeomorphisms, leading to a generalisation of the
Eells-Kuiper invariant. In the next subsection we use that in the general
definition of the $\xi$-invariant of a \gtstr. 

Let $M$ be a closed spin 7-manifold as usual. We can associate to it the
invariants $p_M$, $b_M$ and $q_M^\circ$, where $b_M$ is the torsion
linking form on $\Tor H^4(M)$, and $q_M^\circ$ is a ``family of quadratic
refinements'' of $b_M$ \cite[\S 2.4]{7class}.
Group isomorphisms $F$ act naturally on these objects by pull-backs, \eg
$F^\#p_M$ is simply $F^{-1}(p_M)$.
For any spin diffeomorphism $f$ of $M$,
the induced action $f^*$ on $H^4(M)$ preserves these invariants, \ie
\[ (f^*)^\#p_M = p_M, \quad (f^*)^\#b_M = b_M, \quad
(f^*)^\#q^\circ_M = q^\circ_M ; \]
in fact, this remains true even if $f$ is merely a homeomorphism
\cite[Theorem 1.2]{7class}. We define a function 
\[ P : \aut(H^4(M), b_M, p_M) \to \Z/2\dpi{M}\Z \]
as follows \cite[(39)]{7class}.
Let
\begin{equation}
\label{eq:divsp}
\divsp := \{ k \in H^4(M) : p_M - d_\pi k \textrm{ is torsion} \},
\end{equation}
pick $k \in \divsp$,
let $t := F(k) - k$, and 
\[ P(F) := d_\pi^2 b_M(t, t) - 2d_\pi \, b_M(p_M {-} d_\pi k, \, t)
\mmod 2d_\pi(M)\Z . \]
Then \cite[Prop 4.16]{7class} states that
\begin{equation}
\label{eq:P}
p^2(f) = P(f^*) \mmod 2d_\pi(M) .
\end{equation}
Meanwhile \cite[(42)]{7class} states that
\begin{equation}
\label{eq:imP}
\im P = 2^r \mdiv(M) \Z/2d_\pi(M) \Z ,
\end{equation}
for some $r(H^4(M), b_M, p_M) \in \{0,1,2\}$,
and $r = 1$ unless $H^4(M)$ has 2-torsion.
Combined with \eqref{eq:div2} this implies \eqref{eq:rdiff}, and hence
Theorem \ref{thm:spinauto_LB}.

Further, if $M$ is 2-connected then \cite[Proposition 3.10]{7class} states that
there exist $f \in \ADiff_{Spin}(M)$ with $f^* = \Id$ on $H^4(M)$
and $p^2(f) = 2\td_\pi n$ for any $n \in \Z$; as in the introduction,
$\td_\pi := \lcm(4, d_\pi)$. It is well known that
$f$ is pseudo-isotopic to a diffeomorphism if $p^2(f)$ is divisible by 224
\cite[Lemma 3.7(iii)]{7class}, so one can find $f \in \DiffSpin(M)$ such
that $p^2(f) = \lcm(224, 2\td_\pi)$. Hence equality holds in \eqref{eq:rdiff},
completing the proof of Theorem \ref{thm:rdiff}
(and hence Theorem \ref{thm:spinauto_UB}).

A key step in the above argument is that $p^2(f) \mmod 2d_\pi$ can be
determined purely algebraically, from the action $f^*$ on $H^4(M)$.
A related fact is that $p^2(f) \mmod 2\td_\pi$ can be determined by the action
of $f^*$ on \emph{Gauss refinements} associated to spin coboundaries of $M$. 
A Gauss refinement on $M$ is a function
\[ g : \divsp \to \Q/\tdf\Z \]
whose mod $\Z$ reduction is determined by the quadratic linking family $q_M^\circ$
and which satisfies
\[ g(k+t) - g(k) =
\frac{d_\pi^2 b_M(t, t) - 2d_\pi \, b_M(p_M {-} d_\pi k, \, t)}{8}
\mmod \textstyle\frac{d_\pi(M)}{4}\Z. \]
For our present purposes the significance of these conditions is that the
difference between two Gauss refinements of $M$ is just a constant
in $\Z/\tdf\Z$, and if $f \in \ADiff_{Spin}(M)$ then
\begin{equation}
\label{eq:gaussaction}
g - (f^*)^\# g = \frac{p^2(f)}{8} \mod \frac{\td_\pi}{4} ,
\end{equation}
where $F^\#g := g \circ F$ for any isomorphism $F$ of $H^4(M)$.

Let $W$ be a 3-connected coboundary of $M$, and $j : H^4(W) \to H^4(M)$ the
restriction map. We can associate a Gauss refinement $g_W$ to $W$ by setting
\[ g_W(jn) := \frac{ (p_W - d_\pi n)^2 - \sign(W)}{8} \mod \tdf \Z \]
for any $n \in H^4(W)$ such that $jn \in \divsp$ (then the image of
$p_W - d_\pi n \in H^4(W;\Q)$ has compact support, and its cup-square is
well-defined in $H^8_{cpt}(W;\Q) \cong \Q$) \cite[(18)]{7class}.
The key property of $g_W$ is that if $f : M \to M'$ is a diffeomorphism and
$W'$ is another 3-connected coboundary of $M'$ then
\begin{equation}
\label{eq:gauss}
g_{W'} - (f^*)^\#g_W =
\frac{p_X^2 - \sign(X)}{8} = 28\ahat(X) \mod \tdf ,
\end{equation}
where $X$ is the closed spin manifold $(-M) \cup_f M'$ \cite[(24)]{7class}.
In particular, the mod 28 reduction of $g_W$ is independent of the choice of
spin coboundary $W$. This defines a generalisation of the Eells-Kuiper
invariant,
\[ \mu_M : \divsp \to \Q/\gcd(28, \tdf)\Z , \]
which distinguishes between
$\gcd\left(28, \Num\big(\frac{2^r \mdiv}{8}\big) \right)$ different smooth
structures on the 
topological manifold underlying $M$ \cite[Corollary 4.14]{7class}.
Together with the homeomorphism 
invariants $(H^4(M), q_M^\circ, p_M)$, it
classifies 2-connected 7-manifolds up to diffeomorphism.

\begin{thm}[{\cite[Theorem 1.2]{7class}}]
\label{thm:classification}
For any pair of $2$-connected $7$-manifolds $M_0$ and $M_1$,
an isomorphism $F \colon H^4(M_1) \to H^4(M_0)$
is realised by a diffeomorphism $f \colon M_0 \cong M_1$ such that $F = f^*$
if and only if
$(q^\circ_{M_1}, \EK_{M_1}, p_{M_1}) = F^\#(q^\circ_{M_0}, \EK_{M_0}, p_{M_0})$.
\end{thm}

\subsection{The $\xi$-invariant}
\label{subsec:xi}

We now give the definition of the $\xi$-invariant of a \gtstr{} $\varphi$, which is a
function $\xi(\varphi) : \divsp \to \Q/3\td_\pi\Z$ (with $\divsp$ as in
\eqref{eq:divsp}).
We also explain how the pair $(\nu,\xi)$ distinguishes between
$24 \NumB{\frac{2^r\dinfty{M}}{224}}$ deformation-equivalence classes on a spin
7-manifold $M$. This entails Theorem \ref{thm:spinauto_LB}, and when $M$ is
2-connected Theorem \ref{thm:spinauto_UB} implies Theorem \ref{thm:nuxi},
that $(\nu,\xi)$ is a complete invariant of $\GTD(M)$.

\begin{defn}
\label{def:xi}
Let $\varphi$ be a \gtstr{} on a closed 7-manifold with
$Spin(7)$-coboundary~$W$. The $\xi$-invariant of $\varphi$ is
the function
\[ \xi(\varphi) := 7(\chi(W) - 3\sign(W)) +12 g_W \; : \;
\divsp \to \Q/3\td_\pi \Z . \]
\end{defn}
Combining \eqref{eq:defect} and \eqref{eq:gauss} shows that $\xi(\varphi)$ is
diffeomorphism-invariant (and in particular independent of the choice of $W$):
if $f : M' \to M$ is a diffeomorphism then
\[ (f^*)^\#(\xi(f^* \varphi)) = \xi(\varphi) . \]
The relations
\begin{align*}
2D(\varphi, \varphi') &= \nu(\varphi') - \nu(\varphi) \mod 48 , \\
14D(\varphi, \varphi') &= \xi(\varphi') - \xi(\varphi) \mod 3\td_\pi,
\end{align*}
for \gtstr s $\varphi$ and $\varphi'$ on the same $M$
mean that $(\nu,\xi)$ determine
$D$ mod $\lcm(24, \Num\big(\frac{3\td_\pi}{14}\big))$.
Moreover, these relations help us see that precisely
$\lcm(24, \Num\big(\frac{3\td_\pi}{14}\big)) = 24 \NumB{\frac{d_\pi}{112}}$
pairs $(\nu,\xi)$ are realised, namely the ones satisfying
\begin{subequations}
\begin{align}
\nu &= \chi_\Q(M) \mod 2 ,\\
\label{eq:recover_ek}
\frac{\xi - 7\nu}{12} &= \mu_M \mod \gcd(28,\tdf) .
\end{align}
\end{subequations}
However, this does not mean that there are $24 \NumB{\frac{d_\pi}{112}}$
different deformation-equivalence classes, as one has to take into account that
$f \in \DiffSpin(M)$ acts non-trivially on Gauss refinements and hence
on $\xi$: \eqref{eq:gaussaction} implies
\[ \xi(f^*\varphi) - \xi(\varphi) =  \frac{3}{2} p^2(f) \mod 3\td_\pi. \]
Using Theorem \ref{thm:rdiff}, the mod $2^{r-1}3\dinfty{M}$ reductions of
$\xi$ of deformation-equivalent \gtstr s on $M$ must still be equal.
Hence we can use $(\nu,\xi)$ to distinguish between at least
\[ \lcm \left(24, \NumB{\frac{2^{r-1}3\dinfty{M}}{14}}\right) =
24 \NumB{\frac{2^r\dinfty{M}}{224}} \]
deformation-equivalence classes. For 2-connected $M$ this is precisely
the number of deformation-equivalence classes on $M$ according to
Theorem \ref{thm:spinauto_UB}, so $(\nu,\xi)$ distinguishes between all the
classes, completing the proof of Theorem \ref{thm:nuxi}.

Given a \gtstr, we can use \eqref{eq:recover_ek} to recover the Eells-Kuiper
invariant of the underlying smooth manifold from $(\nu, \xi)$.
Hence Theorem \ref{thm:classification} implies that we can classify
closed 2-connected manifolds with homotopy classes of \gtstr{} using the quintuple
$(H^4(M), q_M^\circ, p_M, \nu, \xi)$.

\begin{thm}
\label{thm:classify_g2}
Let $M_i$ be closed 2-connected 7-manifolds, and $\varphi_i$ \gtstr s on $M_i$.
Given an isomorphism $F : H^4(M_1) \to H^4(M_0)$, there is a diffeomorphism
$f : M_0 \to M_1$ such that $F = f^*$ and $f^* \varphi_1$ is homotopic to
$f^* \varphi_0$ if
and only if $\nu (\varphi_0) = \nu(\varphi_1)$ and
$F^\#(p_{M_0}, q_{M_0}, \xi(\varphi_0)) = (p_{M_1}, q_{M_1}, \xi(\varphi_1))$.
\end{thm}

\bibliographystyle{amsinitial}
\bibliography{g2geom}

\end{document}